\documentclass[11pt,reqno]{amsart}

\usepackage[shortlabels]{enumitem}
\usepackage{esint}
\usepackage{physics}
\usepackage{relsize}
\usepackage[backref]{hyperref}
\usepackage{mathtools}
\usepackage{amsfonts}
\usepackage[all]{xy}
\usepackage{geometry}
\usepackage{stmaryrd}
\usepackage{tikz-cd}
\usepackage[b]{esvect}
\newcommand{\p}{\partial}

\newcommand{\C}{\mathbb C}

\newtheorem{theorem}{Theorem}[section]
\newtheorem{lemma}[theorem]{Lemma}
\newtheorem{definition}[theorem]{Definition}
\newtheorem{example}[theorem]{Example}

\theoremstyle{corollary}
\newtheorem{conjecture}[theorem]{Conjecture}
\theoremstyle{conjecture}

\theoremstyle{assumption}

\newtheorem{question}[theorem]{Question}

\theoremstyle{proposition}
\newtheorem{proposition}[theorem]{Proposition}
\theoremstyle{remark}
\newtheorem{remark}[theorem]{Remark}
\numberwithin{equation}{section}

\DeclareMathOperator{\Id}{Id}
\DeclareMathOperator{\Val}{Val}
\DeclareMathOperator{\Spec}{Spec}
\DeclareMathOperator{\Proj}{Proj}

\newcommand{\ols}[1]{\mskip.5\thinmuskip\overline{\mskip-.5\thinmuskip {#1} \mskip-.5\thinmuskip}\mskip.5\thinmuskip}
\newcommand{\ii}{\ensuremath{\sqrt{-1}}}
\newcommand{\pp}{\bar\partial}
\newcommand{\oo}{\mathcal O}
\newcommand{\ca}{\mathcal C}

\begin{document}

\title{K\"ahler-Ricci shrinkers and Fano fibrations}

\author{Song Sun}
\address{Institute for Advanced Study in Mathematics, Zhejiang University, Hangzhou 310058, China}
\address{
Department of Mathematics, University of California, Berkeley, CA 94720, U.S.A.} 
\email{songsun1987@gmail.com}

\author{Junsheng Zhang}
\address{Simons Laufer Mathematical Sciences Institute\\
17 Gauss Way\\ Berkeley, CA, USA, 94720\\}
\curraddr{}
\email{jszhang@berkeley.edu}
\thanks{This work is completed at the SLMath (MSRI) during Fall 2024 (partially supported by the NSF Grant DMS-1928930). We thank Yuji Odaka for the helpful discussion and comments. The second author also thanks the IASM at Zhejiang University for the hospitality during his visit in Spring 2024, where part of this work was done. The authors thank Carlos Esparza for sharing his draft and Yu Li for bringing to our attention the question of the simply-connectedness of Kähler-Ricci shrinkers.}

\begin{abstract}
In this paper, we build connections between  K\"ahler-Ricci shrinkers, i.e., complete (possibly non-compact) shrinking gradient K\"ahler-Ricci solitons, and algebraic geometry. In particular, we 
\begin{itemize}[leftmargin=*]
\item prove that a K\"ahler-Ricci shrinker is naturally a quasi-projective variety, using birational algebraic geometry;
\item  formulate a conjecture relating the existence of K\"ahler-Ricci shrinkers and K-stability of polarized Fano fibrations, which unifies and extends the  YTD type conjectures for K\"ahler-Einstein metrics, Ricci-flat K\"ahler cone metrics and compact K\"ahler-Ricci shrinkers;
\item   formulate conjectures connecting tangent flows at singularities of K\"ahler-Ricci flows and algebraic geometry, via a 2-step degeneration for the weighted volume of a Fano fibration.
\end{itemize}
 \end{abstract}

\maketitle

\section{Introduction}

A K\"ahler-Ricci shrinker (i.e., a complete shrinking gradient K\"ahler-Ricci soliton)  consists of the data $(X, J, \omega, f)$, where $(X, J,\omega)$ is a complete K\"ahler manifold, and $f$ is a real-valued function on $X$ such that $J\nabla f$ is a holomorphic Killing field and the following equation holds
\begin{equation}\label{e:1.1}
	\mathrm{Ric}(\omega)+\sqrt{-1}\p\bar\p f=\omega.
\end{equation}
The vector field $J\nabla f$ is referred to as the \emph{soliton vector field}.

 Our first goal in this paper is to prove the following general result, which is a first step towards connecting K\"ahler-Ricci shrinkers with algebraic geometry.

\begin{theorem}\label{t:main1}
	A K\"ahler-Ricci shrinker is naturally a quasi-projective variety.
\end{theorem}
It follows from the proof that a K\"ahler-Ricci shrinker naturally defines a \emph{polarized Fano fibration} (see  Definition \ref{d:polarized Fano}). Theorem \ref{t:main1} answers an open question in \cite[Section 7.2]{ConlonDS}. Unlike typical approaches on similar compactification problems that use holomorphic functions and sections via analysis, which often require extra geometric assumptions at infinity,  our proof uses instead birational algebraic geometry, specifically Birkar's boundedness theorem for Fano-type varieties \cite{Birkar2}.

Given Theorem \ref{t:main1}, the second goal of this paper is to make precise a general conjecture relating the existence of K\"ahler-Ricci shrinkers and stability in algebraic geometry.
\begin{conjecture}\label{c:1-1}
	A polarized Fano fibration $(\pi: X\rightarrow Y, \xi)$ admits a K\"ahler-Ricci shrinker, which is unique up to the action of $\mathrm{Aut}(X, \xi)$ (the group of automorphisms of $X$ preserving $\xi$),  if and only if it is K-polystable. \end{conjecture} 
The notion of K-stability for polarized Fano fibrations will be defined in Section \ref{s:K-stability}. In general, we allow $X$ to have singularities and thus, an appropriate notion of a \emph{weak} solution is needed in the sense of pluripotential theory; see Section \ref{sec--ytd type} for more detailed discussion on this. Conjecture \ref{c:1-1} sets a uniform ground for the Yau-Tian-Donaldson type conjectures for Fano type situations studied previously, including positive K\"ahler-Einstein metrics, Ricci-flat K\"ahler cones, and compact K\"ahler-Ricci shrinkers.  Along the way, we introduce a local invariant, the \emph{weighted volume} $\mathbb W$, for Fano fibrations. This unifies the normalized volume in \cite{Li2015} for a klt singularity and the $\tilde \beta$-invariant in \cite{HL2024} for a Fano variety; see the discussion at the end of Section \ref{s:weighted volume}.

The Yau-Tian-Donaldson type conjecture in general predicts that the existence of canonical K\"ahler metrics is equivalent to certain stability notions in algebraic geometry. The aim is to relate the question of solving a difficult non-linear partial differential equation to more amenable algebro-geometric numerical conditions. 

For K\"ahler-Einstein metrics on Fano manifolds the Yau-Tian-Donaldson conjecture was proved in \cite{CDS}, which provides an algebro-geometric criterion in terms of K-polystability. The easier direction that the existence of K\"ahler-Einstein metrics implies K-polystability was previously known by  \cite{Berman2016} for general Fano varieties.   
One follow-up research direction focused on extending the result of \cite{CDS} to Fano varieties.  This was established for $\mathbb Q$-Gorenstein smoothable Fano varieties by \cite{SSY}. Based on the variational approach initiated in \cite{BBJ},  it is established in \cite{Li2022} the case of general Fano varieties satisfying a stronger notion of reduced uniform K-stability. The algebro-geometric study of K-stability of Fano varieties has seen many substantial progress recently, using deep techniques from birational geometry such as  \cite{HMX} and  \cite{Birkar19, Birkar2021}. In particular, it is proved in \cite{LXZ} that K-polystability of a Fano variety is equivalent to reduced uniform K-stability, which then completes the algebro-geometric aspect for general Fano varieties. On the differential geometric side however, there are still intriguing questions on the geometric behavior of the singular K\"ahler-Einstein metrics and linking them to metric geometry (for example RCD spaces); there are important recent progress \cite{GPSS, Szekelyhidi}.
For Ricci-flat K\"ahler cone metrics on Fano cones (or Sasaki-Einstein metrics), the YTD type conjecture, formulated in \cite{CollinsSz2015}, was proved in \cite{CollinsSz} for isolated singularities, and in \cite{Li2021} for general Fano cones.

K\"ahler-Ricci shrinkers are natural generalizations of both positive K\"ahler-Einstein metrics (where $f=0$) and of Ricci-flat K\"ahler cones (where $\mathrm{Ric}=0$). In the compact setting, it was proved by \cite{TianZ1} that the soliton vector field is a holomorphic invariant of the Fano manifold which minimizes a strictly convex functional and it was proved in \cite{TianZ2,TianZ1} that K\"ahler-Ricci shrinkers when exist are unique up to automorphisms. A YTD type conjecture for compact K\"ahler-Ricci shrinkers is known to hold due to the work of \cite{BermanW} and  \cite{DaSz} (see also \cite{CSW}). Generalizations to singular Fano varieties are proved by \cite{BermanW, HL2023, BLLZ2023}.

In the non-compact setting a systematic study of K\"ahler-Ricci shrinkers was initiated in \cite{ConlonDS}. In particular, it was realized, using the Duistermaat-Heckman localization formula in symplectic geometry, that the soliton vector field still can be determined by a weighted volume minimization principle in suitable settings. This reveals the ``Fano" aspects of non-compact K\"ahler-Ricci shrinkers, which behave better than the more extensively studied case of Ricci-flat K\"ahler metrics. It also allows some partial classification theory of K\"ahler-Ricci shrinkers. 
In \cite{ConlonDS} a preliminary attempt was made to relate the existence of K\"ahler-Ricci shrinkers to algebro-geometric stability. A complete classification of non-compact K\"ahler-Ricci shrinkers in complex dimension 2 is known, due to the works in \cite{ConlonDS, CCD, BCCD, LiWang}.  Conjecture \ref{c:1-1} incorporates the YTD type conjecture for all the above special cases.

The third goal of this paper is to describe a picture for understanding singularities of K\"ahler-Ricci flow space-time in terms of local algebraic geometry of Fano fibrations (see Definition \ref{def:Fano fibration}), extending the  local singularity theory in \cite{DS2}. The general scheme is to connect tangent cones of singularities in canonical differential geometric objects with algebraic geometry. Tangent cones, arising from rescalings and often showing natural symmetries, are fundamental in studying singularities in geometric PDEs. However, their uniqueness, independent of rescaling subsequences, remains a well-known open problem. In \cite{DS2}  the uniqueness of tangent cones of non-collapsed singularities of polarized K\"ahler-Einstein metrics was established by drawing connections with algebraic geometry. More importantly, \cite{DS2} gives an algebro-geometric description of the unique tangent cone through a \emph{2-step degeneration theory} in terms of a valuation canonically defined by the metric. The latter is conjectured in \cite{DS2} to be an algebro-geometric invariant and to lead to a local stability theory for algebraic singularities. \cite{Li2015}
related this picture to a generalized volume minimization in Sasaki geometry and reformulated this
conjecture using more algebro-geometric terminologies. The conjecture was confirmed in \cite{LiXu2018,LWX} and the algebro-geometric aspect, referred to as the stable degeneration conjecture, is further extended and developed by many people (see for example \cite{Blum, BlumLiu, Xu2020, Xu-Zhuang2021, XZ, Xu-Zhuang2024} and the references therein). 

A parabolic generalization of this picture is the setting of singularities of non-collapsed K\"ahler-Ricci flow space-time.  Natural examples include finite time singularities and ancient solutions of K\"ahler-Ricci flow and the static space of a singular Ricci-flat K\"ahler metric. The corresponding tangent flows are exactly K\"ahler-Ricci shrinkers (with possible singularities). 
On a Fano manifold, it was discovered by  \cite{CSW} that the long-time behavior of the \emph{normalized} K\"ahler-Ricci flow, when the initial class is proportional to the canonical class, generates new algebro-geometric degenerations for Fano manifolds. In particular, it was proved that any long-time solution to the normalized Ricci flow on a Fano manifold leads to a 2-step degeneration theory (similar to \cite{DS2}) to a unique K\"ahler-Ricci shrinker on a Fano variety. It was conjectured in \cite{CSW} that this degeneration is independent of the initial data hence inducing an optimal degeneration which is an algebro-geometric invariant of a Fano manifold. This conjecture was confirmed by \cite{HL2024} for smooth Fano manifolds, based on earlier work of \cite{He} and \cite{DeSz}. Further algebro-geometric aspect of the conjecture on general Fano varieties was proved in \cite{BLLZ2023}.

In this paper we formulate a series of conjectures that would yield the uniqueness of tangent flows of non-collapsed K\"ahler-Ricci flow space-time, in terms of a 2-step degeneration picture for the \textit{weighted volume} minimization; see Section \ref{section--conjectures}. It is very likely that the deep work of \cite{Birkar19,Birkar2021,Birkar,Birkar2} combined with
suitable extensions of the previous techniques \cite{Xu2020,BLX2022,HL2023,BLLZ2023,LWX,XZ} could be used to solve the algebro-geometric aspects of these conjectures. Moreover, we expect the \textit{weighted volume}, as a local invariant for Fano fibrations, to play a role in the study of boundedness questions in birational geometry. We leave these for future study.

\section{Algebro-geometric set-up}
In this section, we introduce the notions of a Fano fibration and a polarized Fano fibration. These provide the algebro-geometric set-up for our discussion later in this paper. 
\subsection{Fano fibrations}
 In this paper, a variety is always a quasi-projective integral scheme over $\mathbb C$. Unless otherwise stated, a point on variety is assumed to be closed. 
 A \emph{fibration} is by definition a surjective projective morphism $\pi:X\rightarrow Y$ between varieties such that $\pi_*\mathcal O_X=\mathcal{O}_Y$. 
\begin{definition}[Fano fibration]\label{def:Fano fibration}
	 A Fano fibration is a fibration $\pi: X\rightarrow Y$ where $X$ and $Y$ are normal varieties,  $X$ has klt singularities, and $-K_X$ is a relatively ample $\mathbb Q$-Carier divisor.
\end{definition} 
This definition agrees with the one introduced in \cite{Birkar}, where the study of Fano fibrations is motivated by the minimal model program. In particular, contractions of $K_X$-negative extremal faces are natural Fano fibrations. However, we emphasize that here do not require $Y$ to be projective. Indeed, our main interest is to study the \emph{local} behavior around a fixed fiber of a Fano fibration. 

\begin{definition}[Fano fibration germ]
    A Fano fibration germ consists of the data $(\pi: X\rightarrow Y, p)$, where $\pi: X\rightarrow Y$ is a Fano fibration and $p$ is a closed point in $Y$. 
    We call $p$ the center of the Fano fibration germ.
\end{definition}

Since we are interested in the germ of $X$  at $\pi^{-1}(p)$, we often assume $Y$ is affine and we are allowed to shrink $Y$ when necessary. In particular, when $Y=\{p\}$, a Fano fibration germ is simply a Fano variety; when $\pi$ is the identity map a Fano fibration germ is simply a klt singularity.  

\begin{remark}As in the literature, one can define the notion of a \emph{log fano fibration} by including a boundary divisor $D$ on $X$, extending the notion of log Fano varieties and log Fano cones. One can also try to extend the discussion to the complex-analytic setting. For our purpose here we will not discuss these in this paper. 
\end{remark}

\subsection{Polarized Fano fibrations}
\emph{Polarized} Fano fibrations are Fano fibrations endowed with a continuous symmetry.
We first recall the definition of \emph{polarized affine cones} in \cite{CollinsSz}, which provides the algebro-geometric framework for studying canonical metrics in Sasaki geometry.  Throughout this paper, $\mathbb T$ denotes a real compact torus whose Lie algebra is denoted by $\mathrm{Lie}(\mathbb T)$.

\begin{definition}[Polarized affine cone \cite{CollinsSz}]
	A polarized affine cone $(Y, \xi,\mathbb T)$ consists of a normal affine variety $Y=\Spec(R)$ equipped with a $\mathbb T$-action with a unique fixed point and an element $\xi\in \mathrm{Lie}(\mathbb T)$ in the Reeb cone. The Reeb cone of the $\mathbb T$-action is defined as the set of all $\eta\in \mathrm{Lie}(\mathbb T)$  such that under the weight decomposition 
	$$R=\bigoplus_{\alpha\in \mathrm{Lie}(\mathbb T)^*}R_\alpha,$$ 
	we have $\langle \alpha, \eta\rangle>0$ for all non-zero $\alpha$ with $R_\alpha\neq \{0\}$.
\end{definition}

\begin{definition}[Polarized Fano fibration] \label{d:polarized Fano}
	A polarized Fano fibration $(\pi: X\rightarrow Y, \xi)$ consists of a Fano fibration $\pi: X\rightarrow Y$ equipped with an equivariant torus $\mathbb T$-action and an element $\xi\in \mathrm{Lie}(\mathbb T)$, such that $(Y, \mathbb T, \xi)$ is a polarized affine cone. 
	
	We also call $\xi$ the Reeb vector field. The Reeb cone of the polarized Fano fibration is defined to be the Reeb cone of $Y$.
	\end{definition} 

   Notice that $Y$ is determined by $X$ as $\mathrm{Spec}(R_X)$, where $R_X$ denotes the ring of regular functions on $X$. A polarized Fano fibration naturally induces a  Fano fibration germ, with the center given by the vertex of $Y$ (the unique fixed point of the $\mathbb T$-action). By definition it also follows that there are $\mathbb T$-equivariant embeddings of $Y$ into $\C^{N_1}$ and $X$
 into $\C^{N_1}\times \mathbb C\mathbb P^{N_2}$, for some linear actions of $\mathbb T$ on $\C^{N_1}$ and $\mathbb C\mathbb P^{N_2}$, such that $\pi$ is induced by the projection $\pi_1: \C^{N_1}\times \C\mathbb P^{N_2}\rightarrow \C^{N_1}$, and $\pi_2^*\mathcal O(1)|_{X}$ gives $-mK_X$ for some positive integer $m$.

As we shall see in this paper, the notion of a polarized Fano fibration provides the appropriate framework for studying general K\"ahler-Ricci shrinkers. It also incorporates many settings introduced previously. 
First, when $Y$ is a point, by definition the Reeb cone is the entire $\mathrm{Lie}(\mathbb T)$. In this case a polarized Fano fibration is simply a Fano variety $X$ endowed with a holomorphic vector field $\xi$, which we call a \emph{polarized Fano variety}. When $\xi=0$ this is the setting for studying K\"ahler-Einstein metrics.  When $\xi\neq0$ this is the setting for studying K\"ahler-Ricci shrinkers with a prescribed soliton vector field. Secondly, when $\pi$ is the identity map a polarized Fano fibration is the same as a polarized affine cone with klt singularities.  This is usually referred to as  a \emph{Fano cone}. This is the setting for studying the existence of  Ricci-flat K\"ahler cone metrics with a given Reeb vector field.

\section{Algebraicity of K\"ahler-Ricci shrinkers }\label{s:algebraicity of shrinkers}
A compact K\"ahler-Ricci shrinker lives on a Fano manifold, which is projective, so it defines an obvious polarized Fano fibration whose base consists of a single point. In the non-compact case, there is no immediate restriction on the underlying complex manifold -- a priori it may even have infinite topology.  To connect with algebraic geometry we prove the following

\begin{theorem}\label{t:main}
	A K\"ahler-Ricci shrinker $(X, J, \omega, f)$ naturally defines a polarized Fano fibration. In particular, $X$ is quasi-projective. 
\end{theorem}
\begin{remark}When the curvature tensor has quadratic decay the above result was proved in \cite{ConlonDS}, in which case the shrinker is known to be conical.   
\end{remark}

The rest of this section is devoted to the proof of this theorem. 
Let $(X, \omega, f)$ be a non-compact K\"ahler-Ricci shrinker. In this case, by \cite{CaoZhou} the function $f$ is bounded below and proper. 
Denote the vector field $\xi=J\nabla f$. Then $\xi$ is a holomorphic Killing field on $X$ and $\xi\mathbin\lrcorner \omega=-df$. By the Bochner formula one can see that $$\bar \p(\Delta_\omega f+f+\frac{1}{2}(J\xi)\cdot f)=0.$$ Since $f$ is real-valued, we may normalize $f$ so that \begin{equation}\label{normalization of the soliton potential}
	\Delta_\omega f+f+\frac{1}{2}(J\xi)\cdot f=0.
\end{equation}
In particular, we have $\inf_Xf<0$.

The vector field $\xi$ generates a holomorphic isometric action of a compact (real) torus $\mathbb T$ on $X$. In general $\mathbb T$ may have rank bigger than $1$. We first perturb $\xi$ to a rational element $\eta$ in $\mathrm{Lie}(\mathbb T)$ which generates an $S^1$-action with a proper moment map.
 As is explained in \cite[Proposition 5.14]{ConlonDS}, the $\mathbb T$ action has a moment map $\mu: X\rightarrow \mathrm{Lie}(\mathbb T)^*$ such that  each $\eta\in \mathrm{Lie}( \mathbb T)$, viewed as a holomorphic vector field on $X$, satisfies that $\eta \mathbin\lrcorner \omega=-du_\eta$ for $u_\eta:=\langle\mu, \eta\rangle$. Moreover, $u_\eta$ can be normalized by requiring that 
 \begin{equation}\label{e:Hamiltonian equation}
     \Delta_\omega u_\eta+u_\eta+\frac{1}{2}(J\eta)\cdot f=0.
 \end{equation} By definition $f=u_\xi$. Since $f$ is proper and bounded below it follows that $\mu$ is also a proper mapping. By the generalized Atiyah-Guillemin-Sternberg convexity theorem in the non-compact setting due to Hilgert-Neeb-Plank \cite{HNP}, we know that the image $\mu(X)$ is a closed convex and locally polyhedral set in $\mathrm{Lie}(\mathbb T)^*$. 
 
\begin{lemma}\label{l:choice of eta}For $\eta$ sufficiently close to $\xi$, the function $u_\eta$ is proper and bounded below.
\end{lemma}

 \begin{proof} 
 For $\eta\in \mathrm{Lie}(\mathbb T)$ we denote by $L_\eta$ the associated linear function on $\mathrm{Lie}(\mathbb T)^*$. 
 Since $\inf_X f<0$ we can find a point $w_0\in \mu(X)$ with $L_\xi(w_0)=0$. Now for any $w\in \mu(X)$ with $L_\xi(w)>1$, by the convexity of $\mu(X)$ we may write $$w=L_\xi(w) w'+(1-L_\xi(w))w_0$$ for some $w'\in \mu(X)$ with $L_\xi(w')=1$.  Since $f$ is proper the level set $\{L_\xi(w)=1\}\cap \mu(X)$ is compact. Fix a norm $\|\cdot\|$ on $\mathrm{Lie}(\mathbb T)$, we can find $\delta>0$ such that if $\|\eta-\xi\|\leq\delta$ then $\frac{3}{2}>L_\eta(w-w_0)>\frac{1}{2}$ for all $w$ with $L_\xi(w)=1$.
 For such $\eta\in \mathrm{Lie}(\mathbb T)$, we then have 
 $$L_\eta(w)=L_\eta(w'-w_0)L_\xi(w)+L_\eta(w_0)\geq \frac{1}{2}L_\xi(w)+L_\eta(w_0).$$
 In particular, $L_\eta$ is proper and bounded below on $\mu(X)$. The conclusion follows. 
 \end{proof}

  We fix an element $\eta$ given in the above lemma such that it generates an effective $S^1$ action and denote $u:=u_\eta$. Then $\eta=J\nabla u$. It is well-known that $u$ is a Morse-Bott function with only even indices. Since $u$ is proper, we then conclude that each of its level sets is connected \cite[Theorem 4.1]{HNP}.  For a regular value $z$ of $u$, we do K\"ahler reduction at $u=z$. The result is a K\"ahler  metric with orbifold singularities $\omega_z$ on a polarized  log  pair $(X_z, D_z, L_z)$. Here $X_z$ is a normal complex-analytic variety with quotient singularities; the effective divisor $D_z$ on $X_z$ is given by \begin{equation}
     D_z=\sum_\alpha (1-\frac{1}{r_{z,\alpha}})D_{z, \alpha},
 \end{equation}where $D_{z, \alpha}$'s correspond to complex codimension 2 loci (in $X$) of the set of points with non-trivial stabilizer, and  $r_{z, \alpha}$ is the order of the isotropy along a generic point of $D_{z, \alpha}$; $\omega_z$ is a K\"ahler metric on $X_z$ with orbifold singularities along the singular set of $X_z$ and with cone angle $2\pi/r_{z, \alpha}$ along $D_{z, \alpha}$. In particular, $(X_z, D_z)$ is a klt pair. There is also an orbifold hermitian line bundle $L_z$ over $X_z$ associated with the $S^1$ action, which can also be viewed as a $\mathbb Q$-Cartier divisor on $X_z$. The induced hermitian metric on $L_z$ has Chern curvature form $-\sqrt{-1}\p_z\omega_z$.

 \begin{lemma}
   The following equation holds in the sense of currents
 \begin{equation} \label{e: orbifold reduced Ricci curvature}
	Ric(\omega_{z})=2\pi \delta_{D_z}+\omega_z-z\p_z\omega_z+\sqrt{-1}\partial\bar\partial (-f+\log |\nabla u|^2), 
\end{equation}
where $\delta_{D_z}$ denotes the $(1,1)$-current defined by integration along $D_z$. 
 \end{lemma}

\begin{proof}
This follows from a direct computation  (see for example by \cite{BurnsG} or by following \cite[Section 2.1]{SunZh19}). For the convenience of readers, we include a proof here. 
     Since the result is local it suffices to consider the locus where the $S^1$ action is free. 
 Denote by $\{w_1, \cdots, w_{n-1}\}$ the local holomorphic coordinates on the  complex quotient $R$. Then they can be viewed as local holomorphic functions on $X$.  Let $t$ be an arbitrary local function on $X$ with $\eta(t)=1$,  Then $\{z=u_\eta, t, w_1, \cdots, w_{n-1}\}$ gives a local coordinate system on $X$ and we have $\eta=\p_t$.
 
One can write  $Jdz=h^{-1}\Theta=h^{-1}(dt-\theta),$
where $h>0$ is a local function depending on $z, w_1, \cdots, w_{n-1}$, and  $\theta$ is a family of local 1-forms on $R$ parameterized by $z$. One can see that $h^{-1}=\omega(\eta, J\eta)=|\nabla u|^2$.

We can write the K\"ahler form $\omega$ on $X$ as 
\begin{equation}\label{kahler form into horizontal and vertical}
\omega=dz\wedge \Theta+\omega_z,
\end{equation}
where $\omega_z$ is a family of $(1,1)$-forms  on $R$ parameterized by $z$. The condition that $\omega$ is K\"ahler is equivalent to 
\begin{equation}\label{omegahequation}
\begin{cases}\p_z^2\omega_z+d_Rd_R^ch=0\\
d\Theta=\p_z\omega_z-dz\wedge d_R^ch.
\end{cases}
\end{equation}
where $d_R$ denotes the differential along $R$. Locally we can write $$\omega_z^{n-1}=V(\sqrt{-1})^{(n-1)^2}\Omega_R\wedge\bar\Omega_R,$$ where $\Omega_R=dw_1\wedge\cdots\wedge dw_{n-1}$ is the holomorphic volume form on $R$ and $V$ is a positive function. Then 
\begin{equation}\label{ricci curvature for quotient}
   \mathrm{Ric}(\omega_z)=-\frac{1}{2}d_Rd_R^c\log V. 
\end{equation}On $X$ we have a local  holomorphic volume form $\Omega=\kappa\wedge\Omega_R$ where $\kappa$ is the dual $(1,0)$-form to the holomorphic vector field  $\eta^{1,0}:=\frac{1}{2}(\eta-\sqrt{-1}J\eta)$. One can see that $\kappa=(\Theta-\ii hdz+\kappa')$ where $\kappa'$ is a form on $R$. Then we know 
\begin{equation}\label{global Ricci curvature} \frac{\omega^n}{h^{-1}V\Omega\wedge\ols\Omega}\equiv\mathrm{const}, \text{ and }
\mathrm{Ric}(\omega)=-\frac{1}{2}dd^c\log(h^{-1}V).
\end{equation}
Denote $F=f-\log(h^{-1}V)$, then by \eqref{global Ricci curvature} and \eqref{e:1.1},
 we obtain \begin{equation}
    \omega=\frac{1}{2}dd^c(F)=\frac{1}{2}\left(
  d_Rd_R^cF+dz\wedge d_R^c(\p_zF)+d(h^{-1}\p_zF)\wedge \Theta+ h^{-1}\p_zFd\Theta\right).
\end{equation}
Then by \eqref{kahler form into horizontal and vertical}, \eqref{omegahequation} and \eqref{ricci curvature for quotient}, we arrive at 
\begin{equation}
    2\omega_z=d_Rd_R^c(f+\log h)+2\mathrm{Ric}(\omega_z)+h^{-1}\p_zF \p_z\omega_z
\end{equation}
and 
\begin{equation}
    \p_z(h^{-1}\p_z F)=2.
\end{equation}
Hence for some constant $C$, we have 
\begin{equation} 
	\mathrm{Ric}(\omega_{z})=\omega_z-(z+C)\p_z\omega_z+\sqrt{-1}\partial\bar\partial (-f+\log |\nabla u|^2).
\end{equation}
It suffices to show $C=0$.  Notice that 
\begin{equation}\label{eq--first expression}
    \mathcal L_{J\eta}F=-h^{-1}F_z=-2(z+C).
\end{equation} On the other hand, by the definition of $F$,
\begin{equation}\label{eq--by defintion}
    \mathcal L_{J\eta}F=J\eta\cdot f-\mathcal L_{J\eta}\log(h^{-1}V).
\end{equation} Since $\mathcal L_{J\eta}\Omega=\mathcal L_{\eta}\Omega=0$, we have 
\begin{equation}\label{second expression}
    \mathcal L_{J\eta}(\log(h^{-1}V)) =\mathcal L_{J\eta}(\log \frac{\omega^n}{\Omega\wedge\ols{\Omega}})=\frac{\mathcal L_{J\eta}\omega^n}{\omega^n}=-2\Delta_\omega u_\eta.
\end{equation}
Combining \eqref{eq--first expression}, \eqref{eq--by defintion},\eqref{second expression} and Equation \eqref{e:Hamiltonian equation}, we conclude that $C=0$.
\end{proof}
So as a log pair we have 
\begin{equation} \label{e:divisor equation}
    -(K_{X_z}+D_z)+zL_z>0.
\end{equation}
It is well-known that the zero set $Q$ of $\eta$ is a disjoint union of smooth totally geodesic complex submanifolds of $X$. Since $u$ is proper, each component of $Q$ is compact. The following is a crucial result.

  \begin{proposition} \label{p:compact critical set}
      $Q$ is compact. 
  \end{proposition}

We first assume Proposition \ref{p:compact critical set}. Then there exists a $z_*>0$ such that the set $X_{\geq z_*}:=u^{-1}([z_*, \infty))$ does not intersect $Q$. So $[z_*, \infty)$ does not contain a critical value. For $z\geq z_*$ we consider the K\"ahler reduction $(X_z, D_z, L_z, \omega_z)$  at $u=z$.  We may identify $(X_z, D_z, L_z)$ holomorphically with a fixed polarized log pair $(M, D, L)$ for all $z\geq z_*$ and  $\omega_z$ is a  family of conical K\"ahler metrics on $M$ parameterized by $z\geq z_*$. We have $\partial_z[\omega_z]=2\pi c_1(L)$.
 It follows from \eqref{e:divisor equation} that $-(K_M+D)+zL$ is ample   on $M$ for all $z\geq z_*$.  Choosing a rational $z$ we see that $M$ is projective. It also follows that $L$ is nef. 

By the Kawamata base point free theorem we know $L$ is semi-ample, that is, there is a surjective morphism $p: M\rightarrow M'$ and a very ample line bundle $L'$ over $M'$ such that $p^*L'=kL$ for some positive integer $k$. Let $Y'$ be the affine cone associated to $-L'$. Denote by $\mathcal L$ the total space of the line bundle $-kL$. By construction and properties of the K\"ahler reduction, we have natural holomorphic maps $\pi': X_{\geq z_*}\rightarrow \mathcal L\rightarrow Y'$ and the $\mathbb C^*$-action generated by $\eta, J\eta$ descends to a $\mathbb C^*$-action on $Y'$ so that  $\pi'$ is $\mathbb C^*$-equivariant. Using the flow of $-J\eta=\nabla u$ we see $\pi'$ extends to the complement of an analytic subset in $X$, and using the fact that $Y'$ is affine  $\pi'$ extends to a global proper holomorphic map $X\rightarrow Y'$. From the equation \eqref{e:1.1}, the line bundle $-K_X$ admits a hermitian metric with positive curvature form. Since $Y'$ is an affine cone, we can find a small neighborhood $U'$ of the vertex such that $\pi'^{-1}(U')$ is weakly pseudo-convex \cite[Chapter VIII]{Demailly}. Then we can apply H\"ormander's $L^2$ method to construct holomorphic sections of $-mK_X$ and for some $U\subset\joinrel\subset U'$ we get an embedding
\begin{equation}\label{local embedding}
    F: \pi'^{-1}(U)\hookrightarrow U\times \mathbb C\mathbb P^N.
\end{equation}
Notice that we have an induced $S^1$-action on $-K_X$ and by Fourier expansion we may assume the holomorphic sections used to construct the embedding $F$ in \eqref{local embedding} are homogeneous with respect to the $S^1$-action. Then we may assume $F$ is $S^1$-equivariant with respect to the natural $S^1$-actions on $\pi'^{-1}(U)$ and $U$, and some linear $S^1$-action on $\mathbb C\mathbb P^N$. Now the $S^1$-action naturally complexifies to a holomorphic $\C^*$-action, and then we obtain a global holomorphic embedding 
\begin{equation}
    X\hookrightarrow Y'\times \mathbb C\mathbb P^N.
\end{equation}Therefore $X$ admits a quasi-projective variety structure and $\pi'$ is a projective morphism.  Let
 $\pi:X\rightarrow Y$ be the Stein factorization of $\pi'$. Then $Y$ is also an affine variety and the ring of regular functions on $Y$ coincides with that of $X$. Therefore $X$ has no bounded holomorphic functions and then we can also easily show that $-J\xi$ has only positive weights on non-constant holomorphic functions.  
 The $\mathbb T$-action generated by the soliton vector field $\xi$ naturally descends to $Y$, making it into a polarized affine cone, so $\pi$ is a polarized Fano fibration. This completes the proof of Theorem \ref{t:main}.

\

Now we prove Proposition \ref{p:compact critical set}.
 We list the components of $Q$ as $ Q_1, Q_2, \cdots$, such that the corresponding critical values $z_1, z_2,  \cdots$ are in an increasing order. It follows from the equation \eqref{e:Hamiltonian equation} that the minimum value $z_1<0$.

Since $(z_{j-1}, z_j)$ contains a rational number, one sees immediately that $X_z$ is a projective variety for all regular values $z$.
For $z\in (z_{j-1}, z_j)$ the flow of $\nabla u=J\eta$ naturally identifies $(X_z, D_z, L_z)$ with a fixed polarized log pair. 
We need to investigate the birational transformation on $(X_z, D_z, L_z)$ when $z$ crosses a critical value. This fits into the general variation of GIT (see \cite{DoHu, Thaddeus}). In our situation, we can work out the details explicitly.

Given any component of $Q$ on which $u=z_j$,  we have an induced linear $S^1$ action on the holomorphic normal bundle $N$. Suppose it has complex index $m$ and co-index $l$. On the fibers of $N$ the $S^1$ action has positive integer weights $p_{1}, \cdots, p_l$ and negative integer weights $-q_1, \cdots, -q_m$. Denote  $r:= \gcd (q_1, \cdots, q_m)$ and $s:=\gcd(p_1, \cdots, p_l)$. Since the $S^1$ action is effective we have $\gcd(r, s)=1$.
Equation \eqref{e:Hamiltonian equation} implies the following weight identity 
\begin{equation}
	\sum_{\alpha=1}^{l} p_{\alpha}-\sum_{\beta=1}^{m} q_{\beta}=-z_j.
\end{equation}
If the $S^1$ action is quasi-free, i.e. the isotropy group of any point is either trivial or the whole $S^1$, then $p_\alpha$ and $q_\beta$'s are all equal to $1$. This equation implies that $z\leq n$ so $Q$ is compact. 

 In general, the situation is more complicated. 
For simplicity of discussion without loss of generality, we assume there is a unique connected component of the critical set $Q$ on which $u=z_j$ and by abusing notations we will still denote this component by $Q$.
Fix $\epsilon>0$ small such that $z_j$ is the only critical value in $(z_j-\epsilon, z_j+\epsilon)$. We denote by $(X, D, L, \omega_z)$ and $(X', D', L', \omega_z')$ the K\"ahler reduction space for $z\in (z_j-\epsilon, z_j)$ and $z\in (z_j, z_j+\epsilon)$ respectively.

\begin{itemize}[leftmargin=*]
\item (Type I, $m=l=1$) $X$ and $X'$ are naturally isomorphic  and $Q$ sits inside both as a divisor. We have $p_1-q_1=-z_j$.
To compare the log canonical divisor $K_X+D$ and $K_{X'}+D'$ we first investigate the linear model case.

Consider the action of $\C^*$ on $\C^2_{w_1, w_2}$ with weight $(\lambda^{-q_1}, \lambda^{p_1})$. The situation is very similar to the discussion of weighted projective spaces (see for example \cite{RossThomas}),  except we allow the weights of mixed sign (so there is a variation of GIT). For the two quotients here we have global chart $X_-:=\{w_1=1\}/\mathbb Z_{q_1}$ and $X_+=\{w_2=1\}/\mathbb Z_{p_1}$.  So the underlying variety of both is identified with $\mathbb C$ and the boundary divisors are given by $D_-=(1-\frac{1}{q_1})[0]$ and $D_+=(1-\frac{1}{p_1})[0]$. We claim that 
\begin{equation}\label{e:model L change}L_--L_+=\frac{1}{p_1q_1}[0].
\end{equation}
To see this, we choose integers $a, b$ such that $q_1a-p_1b=1$. Since $w_1$ has weight $-q_1$ and $w_2$ has weight $p_1$ on the fibers of $L_\pm^{-}$, the rational function $w_1^{a}w_2^{-b}$ has weight $1$, so it defines a section of both $L_-$ and $L_+$. On $X_-$ it vanishes at $w_2=0$ with order $-b$ on the chart upstairs so it contributes to $-\frac{b}{q_1}[0]$. Similarly on $X_+$ it contributes to $\frac{a}{p_1}[0]$. Adding together yields \eqref{e:model L change}.

In our global setting, the above discussion implies that
\begin{equation}
    D-(1-\frac{1}{q_1})Q=D'-(1-\frac{1}{p_1})Q
\end{equation}
and 
\begin{equation}
    L-L'=\frac{1}{p_1q_1}Q
\end{equation}
In particular, 
\begin{equation}
    K_X+D=K_{X'}+D'+\frac{z_j}{p_1q_1}Q.
\end{equation}

\item (Type II, $m>1$, $l>1$) In this case we have a flipping birational map 
\[
\begin{tikzcd}
 X \arrow[rr,dashed, "{\pi}"] \arrow[dr] && X' \arrow[dl] \\
& Y &
\end{tikzcd}
\]
Here $Y$ is the (singular) K\"ahler quotient at the critical value $z_j$. Since $\pi$ and $\pi^{-1}$ are both defined away from Zariski closed sets of codimension at least 2, we have 
\begin{equation}\label{change of divisors under flips}
    \pi_*D=D', \pi_*K_{X}=K_{X'}.
\end{equation}

\item (Type III, $m>1, l=1$) In this case there is a contraction morphism $\pi: X\rightarrow X'$. The exceptional divisor $E$ is a weighted projective space bundle over $Q$. Each fiber is isomorphic to $\mathbb P(\frac{q_1}{r}, \cdots, \frac{q_m}{r})$. To compare the log canonical divisors we again refer to the model situation first. 

Consider the action of $\C^*$ on $\C^{m+1}_{q_1, \cdots, q_m, p_1}$ with weights $(-q_1, \cdots,-q_m, p_1)$. We have two GIT quotients $X_-=\{(w_1, \cdots, w_{m})\neq 0\}/\C^*$ and $X_+=\{w_{m+1}=1\}/\mathbb Z_{p_1}$. Since $\gcd (p_1, r)=1$ the only possible codimension 1 orbifold loci is given by the reduced divisor $E$ set-theoretically cut out by $\{w_{m+1}=0\}$ in $ X_-$. Notice that $E$ is isomorphic to the weighted projective space $\mathbb P(\frac{q_1}{r}, \cdots, \frac{q_m}{r})$. It is also the exceptional set of the natural contraction map $X_-\rightarrow X_+$. The boundary divisor $D_-=(1-\frac{1}{r})E$ and $D_+=0$. We claim that 
\begin{equation}\label{e:model L-}
    L_-=\frac{1}{p_1r}E.
\end{equation}
To see this we consider the function $w_{m+1}$, which has weight $p_1$ on the fibers of $L_-^{-1}$, so it defines a section of $L_-^{p_1}$. At the same time working in a local chart, say $\{w_1=1\}/\mathbb Z_{q_1}$ of $X_-$, one sees that upstairs $w_{m+1}$ has vanishing order $1$, but the zero set $\{w_{m+1}=0\}$ has an isotropy group $\mathbb Z_r$. So it really defines the divisor $\frac{1}{r}E$. Then \eqref{e:model L-} follows.

In our global case, we then have 
\begin{equation}\label{change of divisor under contraction}
  \pi_*K_X=K_{X'},\quad  D=\pi_*^{-1}D'+(1-\frac{1}{r})E,
\end{equation}
and 
\begin{equation}
    L=\pi^*L'+\frac{1}{p_1r}E. 
\end{equation}
Notice also that $-(K_X+D)+z_j L$ is trivial along the weighted projective fibers of $E$, it follows that
\begin{equation}
    K_X+D=\pi^*(K_{X'}+D')+a E,
\end{equation}
where  $$a=\frac{z_j}{p_1r}=\frac{\sum_\beta q_\beta-p_1}{p_1r}>0.$$

\item (Type IV, $m=1, l>1$)
In this case, there is a contraction morphism $\pi: X'\rightarrow X$.  The exceptional divisor $E$ is a weighted projective space bundle over $Q$. Each fiber is isomorphic to $\mathbb P(\frac{p_1}{s}, \cdots, \frac{p_l}{s})$. Similar to the Type III situation we find that 
\begin{equation}\label{change of D under extraction}
    D'=\pi_*^{-1}D+(1-\frac{1}{s})E,
\end{equation}
and 
\begin{equation}
    L'=\pi^*L-\frac{1}{q_1s}E. 
\end{equation}
Also one can compute
\begin{equation}\label{change of divisor under extraction}
    K_{X'}+D'=\pi^*(K_{X}+D)+a E,
\end{equation}
where $$a=-\frac{z_j}{q_1 s}=\frac{\sum_\alpha p_\alpha-q_1}{q_1 s}<0.$$

\end{itemize}

\begin{remark}
    If we only have Type II and Type III critical values, then as $z$ increases we are essentially running a minimal model program along the $\mathbb Q$-divisor $L$. 
\end{remark}

  \begin{lemma}\label{lem--main}
      There exists a $\delta>0$ such that for all regular value $z>0$ one can find an effective  $\mathbb Q$-divisor $\Delta_z$ on $X_z$ such that the following hold
      \begin{itemize}
          \item The pair $(X_z, D_z+\Delta_z)$ is $\delta$-lc, i.e., the log discrepancy of any divisor over $X$ is at least $\delta$;
          \item $K_{X_z}+D_z+\Delta_z\sim_{\mathbb Q}0$;
          \item $\Delta_z$ is a big divisor. 
      \end{itemize}
  \end{lemma}

\begin{proof}
  We start with a positive regular value $z$ sufficiently close to $0$.  From equation \eqref{e:divisor equation} we see that $-(K_{X_z}+D_z)$ is the limit of ample divisors so it is itself a nef divisor. Since $\min_X u<0$ the volume of $(X_z,\omega_z)$ as $z\rightarrow 0$ does not go to zero. So $-(K_{X_z}+D_z)$ is big.  Since $(X_z,D_z)$ is klt, by Bertini's theorem we may find an effective $\mathbb Q$-divisor $\Delta_z$ on $X_z$ such that the statements hold for some $\delta>0$ small. 

  Next, we prove the properties continue to hold with the same $\delta$ when we cross a critical value $z_j$. There are 4 cases 
  \begin{itemize}[leftmargin=*]
  \item(Type I) We simply let $\Delta'=\Delta+\frac{z_j}{p_1q_1}Q$ and we have $D+\Delta=D'+\Delta'$. Then $(X', D'+\Delta')$ is naturally isomorphic to $(X, D+\Delta)$ and $\Delta'$ is big.
    \item(Type II) We let $\Delta'=\pi_*\Delta$, which is big since $\Delta$ is big. Then by \eqref{change of divisors under flips}, we know that $K_{X'}+D'+\Delta'\sim_{\mathbb Q}0$. By \cite[Lemma 3.38]{KM}, we see the log discrepancy of $(X', D'+\Delta')$ agrees with that of $(X, D+\Delta).$

  \item(Type III) We let $\Delta'=\pi_*\Delta$, which is big since $\Delta$ is big. Then by \eqref{change of divisor under contraction}, we know that $K_{X'}+D'+\Delta'\sim_{\mathbb Q}0$. By \cite[Lemma 2.30]{KM} we see the log discrepancy of $(X', D'+\Delta')$ agrees with that of $(X, D+\Delta).$
  \item(Type IV) We let $\Delta'=\pi^*\Delta-aE$. It is crucial here that $a<0$ so that $\Delta'$ is again big and effective. Then by \eqref{change of divisor under extraction}, we have $K_{X'}+D'+\Delta'\sim_{\mathbb Q}0$. Again by \cite[Lemma 2.30]{KM}, the log  discrepancy of $(X', D'+\Delta')$ agrees with that of $(X, D+\Delta).$
  \end{itemize}
  \end{proof}

An immediate consequence is that the coefficients of $D_z$ are uniformly bounded above by $1-\delta$. So at each codimension 1 orbifold point, the order is uniformly bounded by $\frac{1}{\delta}$.  In particular,  the Type I critical values are uniformly bounded. 

Now we apply a deep result in \cite[Corollary 1.4]{Birkar2021}, which implies that the set of all $X_z$  above belongs to a bounded family of normal projective varieties.  We claim that the order of the orbifold group at every point in the underlying variety $X_z$ is uniformly bounded above. There might be many ways of seeing this, for example, a quick way is to apply the lower semi-continuity result of \cite{BlumLiu} (together with \cite{deFernex-Ein-Mustata}) to get a uniform lower bound on the normalized volume at points of $X_z$.

 Now Proposition \ref{p:compact critical set} follows from

\begin{lemma}
    Given a critical value $z_j$ of Type II, III or IV, there is a quotient singularity in the underlying variety  $X_z$, for a regular value $z$ close to $z_j$, with the orbifold order at least $\frac{\delta^{n-1} z_j}{n-1}$. In particular, its normalized volume is at most $\frac{(n-1)^{n}}{\delta^{n-1}z_j}$.
\end{lemma}

\begin{proof}
   Suppose at $z_j$ the weights given by $-q_1, \cdots, -q_m, p_1, \cdots, p_l$. 
    Then we have $\sum_\beta q_\beta\geq z_j$. In particular there is a $\beta$ with $q_\beta\geq \frac{z_j}{n-1}$. Notice that we need to deal with codimension 1 orbifold loci. Since $\gcd(q_1, \cdots, q_m, p_1, \cdots, p_l)=1$ and by the above discussion the $\gcd$ of any $m+l-1$ weights is at most $\frac{1}{\delta}$,   it is easy to see that on $X_z$ for $z\in (z_j-\epsilon, z_j)$ there is at least an orbifold singularity with order $\delta^{n-1}q_\beta\geq \frac{\delta^{n-1} z_j}{n-1}$. 
\end{proof}

\begin{remark}
 Since the above arguments do not use very detailed geometric properties of $X$, the technique is robust so is likely to extend further. For example, with more effort one expects the same proof works for a suitable class of singular K\"ahler-Ricci shrinkers  (c.f. Section \ref{section--conjectures}).
\end{remark}

    The idea of using Kähler reduction to study Kähler-Ricci shrinkers, while quite natural, appears to be novel. As another application, we show that every smooth Kähler-Ricci shrinker is simply-connected. This is also proved independently by Esparza \cite{Esparza2} using a version of Hirzebruch-Riemann-Roch formula on non-compact manifolds. 

\begin{proposition}
     Every smooth Kähler-Ricci shrinker $X$ is simply connected.
\end{proposition}

\begin{proof}
    If \( X \) is compact, it is a standard result that a Fano manifold is simply connected. If \( X \) is non-compact, it is known that \( X \) has a finite fundamental group \cite{Wylie}. Let \( \phi: \widetilde{X} \to X \) denote the universal cover, where \( \widetilde{X} \) inherits the Kähler-Ricci shrinker structure from \( X \). As before, we can perturb the soliton vector field \( \xi \) to obtain a holomorphic vector field \( \eta \), which generates an effective \( S^1 \)-action with a proper Hamiltonian potential \( u \). This \( S^1 \)-action naturally lifts to \( \widetilde{X} \) with Hamiltonian potential given by \( \phi^*u \).  

We perform Kähler reduction for regular values \( \epsilon \) of \( u \) near 0. As mentioned earlier, since \( u \) is proper, it follows that each of its level sets is connected \cite[Theorem 4.1]{HNP}. We claim that as long as \( \phi: \widetilde{X} \to X \) is a nontrivial covering, the quotient \( (\phi^*u)^{-1}(\epsilon)/S^1 \) is a nontrivial covering of \( u^{-1}(\epsilon)/S^1 \). To prove this, it is sufficient to show that the lift of each \( S^1 \) orbit on \( X \) remains a loop in \( \widetilde{X} \).  
Notice that the \( S^1 \) orbit on \( X \) is generated by the flow of \( J\nabla u \). We can use the negative gradient flow line of \( u \) to contract this \( S^1 \) orbit to a point, showing that this orbit is homotopically trivial in \( X \). By the properties of covering spaces, the lift of each \( S^1 \) orbit on \( X \) to \( \widetilde{X} \) must also remain a loop.  

Let \( X_\epsilon = u^{-1}(\epsilon)/S^1 \). By the proof of Lemma \ref{lem--main}, we know that for sufficiently small \( \epsilon \), the pair \( (X_\epsilon, D_{\epsilon}) \) is klt, and \( -(K_{X_\epsilon} + D_{\epsilon}) \) is big and nef. Then, by \cite[Corollary 1.14]{HM07}, we conclude that \( X_{\epsilon} \) is simply connected. Therefore, the covering map \( (\phi^*u)^{-1}(\epsilon)/S^1 \to u^{-1}(\epsilon)/S^1 \) must be trivial. Consequently, \( \widetilde{X} \to X \) must also be a trivial covering, implying that \( X \) itself is simply connected.  
\end{proof}

Finally, we give an intrinsic description of the algebraic structure of  $X$ in terms of $\xi$.  The ring $R_X$ of regular functions on $X$ is given by 
\begin{equation} R_X=\bigoplus_{\alpha\in \mathrm{Lie}(\mathbb T)^*} R_{X,\alpha}, 
\end{equation}
where $ R_{X,\alpha}$ consists of homogeneous functions with weight $\alpha$. Notice that  $\langle \alpha,\xi\rangle>0$ for any non-zero $\alpha$ with $ R_{X,\alpha}\neq 0$. The polarized affine cone $Y$ is isomorphic to $\Spec(R_X)$.

Similarly,  the space of algebraic sections of $-mK_X$ is given by
$$\mathfrak R_m:=\bigoplus_{\alpha\in \mathrm{Lie}(\mathbb T)^*} \mathfrak R_{m, \alpha}, $$ where $\mathfrak R_{m, \alpha}$ consists of homogeneous sections with weight $\alpha$. There exists a positive constant $C$ such that $ \mathfrak R_{m, \alpha}\neq 0$ only when $\langle \alpha, \xi\rangle \geq -Cm$.
Now the direct sum 
$$\mathfrak R:=\bigoplus_{m\geq 0} \mathfrak R_m$$
is a finitely generated graded algebra over $R_X$, and $X$ is isomorphic to $\Proj_{\Spec(R_X)}{\mathfrak R}$.

\section{Weighted volume of Fano fibration germs} \label{s:weighted volume}

Fix a Fano fibration germ $(\pi:X\rightarrow Y, p)$ with $\dim X=n$. Fix a positive integer $r$ such that $L=-rK_X$ is a Cartier divisor. Since $\pi$ is proper, for any $m\geq 0$, the push-forward $\pi_*\mathcal O(mL)$ is a coherent sheaf on $Y$. Denote by $R_m$ the stalk of $\pi_*\mathcal O(mL)$ at $p$. Then we obtain a graded $\oo_{Y,p}$-algebra
\begin{equation}\label{e:4-1}
	\mathcal R:= \bigoplus_{m\geq0} R_m=\bigoplus_{m\geq0}(\pi_*\oo(mL))_p.
\end{equation}

Consider a real valuation $v$ on the function field $K(X)$ which is centered in $\pi^{-1}(p)$. This means that if $\mathcal O_v$ is the valuation ring of $v$, then there is a unique scheme-theoretic point $c_X(v) \in X$, the \textit{center} of $v$, such that we have a local inclusion of local rings $\mathcal O_{X,c_X(v)}\hookrightarrow \mathcal O_v$ and the Zariski irreducible closed set $\overline{\{c_X(v)\}}$ is contained in $\pi^{-1}(p)$.  Moreover, in the case when $\pi$ is birational a valuation on $K(X)$ centered in $\pi^{-1}(p)$ is the same as a valuation on $K(Y)$ centered at $p$. 

The valuation $v$  induces a valuation on  $\mathcal R$ as follows. Choose a closed point $q\in \overline{\{c_X(v)\}}$ and trivialize $L$ in an affine neighborhood of $q$ in $X$, then we have injective maps: 
\begin{equation}\label{eq--identification}
	R_m\hookrightarrow \oo_{X,q}\hookrightarrow \oo_{X,c_X(v)}\hookrightarrow K(X).
\end{equation}
It is straightforward to show that the induced valuation on $\mathcal R$ is independent of the choices.
The valuation induces a filtration $\mathcal F_v^{\bullet}$ on the graded algebra $\mathcal R$, by setting \begin{equation*}
		\mathcal F_v^\lambda R_{m}:=\{f\in R_m| v(f)\geq \lambda\}.
	\end{equation*}
Then as in \cite{WNystrom,BoucChen2011,BKMS,BHJ} we consider
\begin{equation}\label{eq--sequence of measure}
	\mathrm{DH}_m:=\frac{1}{(rm)^n}\sum_{\lambda\geq 0}\dim_{\mathbb C} (\mathcal F_v^{\lambda rm} R_{m}/\mathcal F^{>\lambda rm}_vR_{m})\delta_{\lambda}.
\end{equation}

\begin{proposition}\label{lem-convergence of measures}
For each $m\geq 1$, $\mathrm{DH}_m$ defines a Radon measure on $[0, \infty)$. Moreover, as $m\rightarrow\infty$, $\mathrm{DH}_m$ converges weakly to a unique limit measure $\mathrm{DH}(v)$, called the Duistermaat-Heckman measure of $v$.
\end{proposition}

This can be proved by following the argument in \cite{BoucChen2011,Cutkosky2013,LM2009,KK08}. For the reader's convenience, we provide a proof in the Appendix. From the proof, we  also see that the measure $\mathrm{DH}(v)$ is independent of the choice of $r$, the positive integer chosen at the beginning of this section to ensure $rK_X$ is Cartier.

We will use the notion of \emph{log discrepancy} $A_X(\cdot)\in [0,\infty]$ for valuations, as defined in \cite[Section 5]{JM2012}.
We only remark here that if $X$ has klt singularities and $v$ is nontrivial, then $A_X(v)$ is strictly positive. Moreover
the definition of log discrepancy has been generalized to valuations on normal varieties \cite{deFhacon,BdFFU}, although here we focus on the case where $X$ has klt singularities.

Let $(\pi: X\rightarrow Y, p)$ be a Fano fibration germ. We denote by $\operatorname{Val}_{\pi, p}^{*}$  the set of valuations $v$ with $A_X(v)<\infty$ and centered in $\pi^{-1}(p)$. 

\begin{definition}[Weighted volume of a valuation]\label{def--weighted volume}
Given  $v\in \operatorname{Val}_{\pi, p}^{*}$, we define its weighted volume by
\begin{equation}
	\Bbb W(v):= e^{A_X(v)}\int_\mathbb Re^{-x}\operatorname{DH}(v) .
\end{equation}
\end{definition}
\begin{remark}
    When we need to emphasize the dependence on the Fano fibration, we will use the notation ${\Bbb W}_{\pi}(v)$. 
\end{remark}
 From the estimate \eqref{estimate for the length} in the appendix, it is easy to see that ${\Bbb W}(v)$ is always finite and 
\begin{equation}\label{integral for each m measure}
    {\Bbb W}(v)=e^{A_X(v)}\lim_{m\rightarrow \infty}\int_{\mathbb R}e^{-x}\mathrm{DH}_m(v).
\end{equation}
We will be interested in minimizing  ${\Bbb W}$ on the space $\Val^*_{\pi, p}$. This is closely related to the study of K\"ahler-Ricci shrinkers and K\"ahler-Ricci flows. We will explain more details in Section \ref{section--conjectures}. 
The above definition incorporates several special situations studied previously:

\begin{example} \label{r:two minimization coincide}
    If $\pi$ is the identity map, then $L$ is trivial and $R_m=\mathcal O_{X, p}$ for all $m\geq 0$. One can check that 
    \begin{equation}
        {\Bbb W}(v)=e^{A_X(v)}\operatorname{vol}(v),
    \end{equation}where 
    \begin{equation}
    	\mathrm{vol}(v):=\lim_{m\rightarrow \infty}\frac{n!}{m^n}\dim_{\mathbb C}(\mathcal O_{X,p}/\mathfrak a_{m}(v)).  \end{equation}Here $\mathfrak{a}_m(v):=\{f\in \oo_{X,p}:v(f)\geq m\}$ denotes the valuation ideal of $v$.
    This is a variant of the normalized volume defined by  Li \cite{Li2015}, but they lead to the same minimization problem. The reason is that the normalized volume is scaling invariant and to minimize ${\Bbb W}(v)$ for all valuations, it is sufficient to minimize it over the slice 
    \begin{equation}
        \{v\in \operatorname{Val}_{\pi, p}^{*}\mid A_X(v)=n\}.    \end{equation} The functional ${\Bbb W}$ has the additional advantage that it is \text{not} scaling invariant and it is possible to obtain the strict uniqueness of a minimizer. 
    \end{example}

\begin{example}
      If $Y$ is a point, then $X$ is a Fano variety. In this case, the Duistermaat-Heckman measure $\operatorname{DH}$ is not normalized to be a probability measure, and we have
      \begin{equation}\label{eq--relation with beta invariant}
          {\Bbb W}(v)=\frac{(-K_X)^n}{n!}e^{\widetilde \beta(v)}
      \end{equation}where $\widetilde \beta$ is the functional introduced  in \cite{HL2024}. Note that on a fixed $X$, the factor $(-K_X)^n$ is  not important, however, we expect that the weighted volume functional ${\Bbb W}$ provides a more natural quantity in studying boundedness questions; see the discussion in Section \ref{sec:Moduli space}.
\end{example}

\section{Futaki invariant and K-stability of polarized Fano fibrations}\label{s:K-stability}
In this section, we formulate a notion of K-stability for general polarized Fano fibrations. A preliminary attempt was given in \cite[Section 7]{ConlonDS} and one can see that the objects considered in \cite[Conjecture 7.5]{ConlonDS} are naturally polarized Fano fibrations. 

 Let $(\pi:X\rightarrow Y,\xi)$ be a polarized Fano fibration. In the following discussion, for the simplicity of notation, we assume $K_X$ is a Cartier divisor, and the general case for $K_X$ being $\mathbb Q$-Cartier can be addressed in a similar manner.
\subsection{Weighted volume of Reeb vector fields}
 By definition a polarized Fano fibration $(\pi:X\rightarrow Y,\xi)$ naturally defines a Fano fibration germ with a center at the vertex $o$ of $Y$. 
  Any element $\eta$ in the Reeb cone induces a valuation in $\Val_{\pi, o}^*$. Then we can define the weighted volume of $\eta$ as the weighted volume of this valuation induced by $\eta$. Moreover, in the following, we explain that this weighted volume is determined by the weights of $\eta$ acting on global sections of $-mK_X$.
  
For a given Reeb vector field $\eta$, there are two associated filtrations on the section ring
\begin{equation*}
	\mathcal R:=\bigoplus_{m\geq 0}R_m:=\bigoplus_{m\geq 0}H^0(X,-mK_X).
\end{equation*}
 Firstly, the $\mathbb T$-action admits a canonical lifting to $-K_X$ (as an automorphism of $X$ can push forward tangent vectors) and hence acts on the section ring $\mathcal R$. Then we have a filtration on $\mathcal R$ induced by the weight decomposition 
$R_m=\bigoplus_{\alpha\in\mathrm{Lie}(\mathbb T)^*} R_{m,{\alpha}}:$
 \begin{equation}
 	\mathrm{wt}_{\eta}(s):=\{\langle\alpha,\eta\rangle\mid s=\sum_{\alpha} s_{\alpha}, s_{\alpha}\neq 0\} \text{ and }  \mathcal F_{\mathrm{wt}_\eta}^{\lambda}R_m:=\{s\in R_m\mid \mathrm{wt}_{\eta}(s)\geq \lambda\}.
 \end{equation}
 % and we obtain an associated filtration
 Secondly, $\eta$ induces a valuation $v_\eta$ on the function field $K(X)$ via the weight decomposition  $K(X)=\bigoplus_{\alpha \in\mathrm{Lie}(\mathbb T)^*}K(X)_{\alpha},$ that is, we have
 \begin{equation}
 	v_{\eta}(f):=\{\langle\alpha,\eta\rangle\mid f=\sum_{\alpha} f_{\alpha} \in K(X), f_{\alpha}\neq 0\}.
 \end{equation}
 As $\eta$ is a Reeb vector field, this valuation $v_{\eta}$ has to be centered in $\pi^{-1}(o)$. As before, by locally trivializing the bundle $-K_X$, we obtain a valuation along with the associated filtration on $\mathcal R$. The following result is a direct consequence of \cite[Lemma 2.18]{LiXu2018}.
 \begin{lemma} \label{p:two valuations}For $s\in R_m$,
 we have \begin{equation}\label{eq--valuations local to global}
 	v_{\eta}(s)=\mathrm{wt}_{\eta}(s)+m A_X(v_{\eta}). 
 \end{equation}
 \end{lemma}
In particular, the filtration associated with $v_{\eta}$ and $\mathrm{wt}_{\eta}$ on $R_m$ only differs by a translation.
 We define the measure
 \begin{equation}\label{shift of the measure}
     \mathrm{DH}(\eta):= {\iota_{A_X(v_{\eta})}}_*(\mathrm{DH}(v_\eta))
 \end{equation}where $\iota_{A_X(v_{\eta})}(x)=x-A_X(v_{\eta})$ denotes the translation on $\mathbb R$. Then we have  
 \begin{equation} \label{e:weighted volume for vector fields}
   {\Bbb W}(\eta):={\Bbb W}(v_\eta)= \int_{\mathbb R}e^{-x}\mathrm{DH}(\eta),
 \end{equation}which by \eqref{eq--valuations local to global} can be computed using the weights of $\mathrm{wt_{\eta}}$ on $\mathcal R$.

\subsection{K-stability of polarized Fano fibrations}
 We first define \emph{special test configurations} for a polarized Fano fibration. 
\begin{definition}[Special test configurations]
     Let $\left(\pi:X\rightarrow Y,\xi\right)$ be a polarized Fano fibration. A special test configuration $(\Pi:\mathcal X\rightarrow \mathcal Y,\xi,\eta)$ of $(\pi:X\rightarrow Y,\xi)$ consists of a commutative diagram, where $\Pi_{\mathcal X}$ and $\Pi_{\mathcal Y}$ are surjective flat morphisms,
   \[
\begin{tikzcd}
 & \mathcal{X} \arrow[ld,"\Pi"']\arrow[dr,"\Pi_{\mathcal X}"] \\
\mathcal{Y}  \arrow[rr,"\Pi_{\mathcal Y}"] && \mathbb C
\end{tikzcd}
\]
satisfying the following conditions:
\begin{itemize}
    \item [(1).]$\eta$ is a holomorphic vector field on $\mathcal{X}$ generating a $\mathbb C^*$-action on $\mathcal{X}$ and $\mathcal Y$ such that $\Pi_{\mathcal X}$ and $\Pi_{\mathcal Y}$ are $\mathbb C^*$-equivariant, where the base $\mathbb C$ is equipped with its standard $\mathbb C^*$-action. Moreover there is a $\mathbb C^*$-equivariant isomorphism
    \begin{equation}\label{eq--identification away from cenral fiber}
\begin{tikzcd}
 & \Pi_{\mathcal{X}}^{-1}(\mathbb C^*) \arrow[ld,"\Pi"']\arrow[dr,"\Pi_{\mathcal X}"] \\
\Pi_{\mathcal{Y}}^{-1}(\mathbb C^*)  \arrow[rr,"\Pi_{\mathcal Y}"] && \mathbb C^*
\end{tikzcd}\simeq
\begin{tikzcd}
 & X\times \mathbb C^* \arrow[ld]\arrow[dr] \\
Y\times \mathbb C^*  \arrow[rr] && \mathbb C^*.
\end{tikzcd}
    \end{equation}
    \item [(2).]Under the isomorphism \eqref{eq--identification away from cenral fiber}, the vector field $\xi$ extends to a holomorphic vector field on $\mathcal X$,  still denoted by $\xi$, which commutes with $\eta$.  
    \item [(3).] $\mathcal{X}$ is $\mathbb Q$-Gorenstein and $(\Pi_0:\mathcal X_0\rightarrow \mathcal Y_0,\xi)$ is a polarized Fano fibration, where $\mathcal X_0$ and $\mathcal Y_0$ denote the central fiber of $\Pi_{\mathcal X}$ and $\Pi_{\mathcal Y}$ respectively.
    \end{itemize}
 \end{definition}
 
\begin{definition}[Product test configuration]
    A special test configuration is said to be a product test configuration, if there is an $\xi$-equivariant isomorphism $\mathcal X\simeq X\times \mathbb C$ and $\eta=\eta_0+z\partial_z$, where $\eta_0$ is a holomorphic vector field on $X$ commuting with $\xi$ and $z$ is the standard coordinate on $\mathbb C$.
\end{definition}

Next, we define the \emph{Futaki invariant} of a special test configuration.

\begin{definition}[Futaki invariant]For a special test configuration $(\Pi:\mathcal X\rightarrow \mathcal Y,\xi,\eta)$ with central fiber $(\Pi_0:X_0\rightarrow Y_0, \xi)$, we define 
\begin{equation}\label{eq--definition of futaki}
    \mathrm{Fut}_{\xi}(\eta):=\left.\frac{d}{d t}\right|_{t=0}{\Bbb W}_{\Pi_0}(\xi+t\eta).
\end{equation}
\end{definition}

\begin{definition}[K-stability] A polarized Fano fibration $(\pi:X\rightarrow Y,\xi)$ is said to be K-semistable, if for any special test configuration, the Futaki invariant $\mathrm{Fut}_{\xi}(\eta)\geq 0$. 

A polarized Fano fibration $(\pi: X\rightarrow Y,\xi)$ is said to be K-polystable if it is K-semistable and a special test configuration has vanishing Futaki invariant if and only if it is a product test configuration.
\end{definition}
\begin{remark}
This unifies the notion of K-stability with respect to special test configurations for a polarized Fano variety (when $Y$ is a point) and a Fano cone (when $\pi$ is the identity map).
As in \cite{CollinsSz2015, BHJ, BLLZ2023}, it is possible to define K-stability for more general test configurations, which is expected to be equivalent to the above definition using only special test configurations. In the case when $Y$ is a point this was proved by  \cite{LX2014} and  \cite{HL2023}. In the case where $\pi$ is the identity map, there has been progress \cite{LWu}.
    Moreover, many other related notions, like the $\delta$-invariant \cite{FO2018,BLLZ2023} and Ding-stability, can be generalized to polarized Fano fibration setting.
\end{remark}

The following lemma, whose proof is given in the appendix, works for any polarized Fano fibration $(\pi:X\rightarrow Y,\xi)$, and ensures that the Futaki invariant in \eqref{eq--definition of futaki} is well-defined. 
% Let $(\pi:X\rightarrow Y,\xi)$ be a polarized Fano fibration and we have a compact torus $\mathbb T$-action for which $\pi$ is equivariant. Moreover $\xi$ is in the Reeb cone of the $\mathbb T$-action and $\eta\in \mathrm{Lie}(\mathbb T)$ and we have the weight decomposition
% \begin{equation}
% \bigoplus_{m\geq 0}H^0(X,-mK_X)=\bigoplus_{m\geq 0}\bigoplus_{\alpha \in\mathrm{Lie}(\mathbb T)^*} R_{m,\alpha}.	
% \end{equation}
 
\begin{lemma}\label{lem--futaki is well-defined}
	For a polarized Fano fibration $(\pi:X\rightarrow Y, \xi)$ and $0\neq\eta\in \mathrm{Lie}(\mathbb T)$, we have
	\begin{equation}\label{eq--volume is differentiable}
     	\left.\frac{d}{d t}\right|_{t=0}{\Bbb W}(\xi+t\eta)=-\lim_{m\rightarrow \infty}\frac{1}{m^n}\sum_{\alpha\in \mathrm{Lie}(\mathbb T)^*}e^{-\langle\frac{\alpha}{m},\xi\rangle}\langle\frac{\alpha}{m},\eta\rangle\dim_{\mathbb C}R_{m,\alpha},
	\end{equation}
 \begin{equation}\label{eq-second order derivative}
     	\left.\frac{d^2}{d t^2}\right|_{t=0}{\Bbb W}(\xi+t\eta)=\lim_{m\rightarrow \infty}\frac{1}{m^n}\sum_{\alpha\in \mathrm{Lie}(\mathbb T)^*}e^{-\langle\frac{\alpha}{m},\xi\rangle}\langle\frac{\alpha}{m},\eta\rangle^2\dim_{\mathbb C}R_{m,\alpha}>0.
 \end{equation}
\end{lemma}

It follows from \eqref{eq-second order derivative} that the weighted volume is a strictly convex function on the Reeb cone. In particular, there exists at most one critical (and minimum) point of ${\Bbb W}$ on the Reeb cone. This fact generalizes the volume minimization principle in \cite{MSY} for Ricci-flat K\"ahler cones, in  \cite{TianZ1} for compact K\"ahler-Ricci shrinkers and in \cite{ConlonDS} for non-compact K\"ahler-Ricci shrinkers with bounded Ricci curvature. 
 \begin{remark}
  Following \cite{BermanW, BLLZ2023,BLQ}, two-dimensional Duistermaat-Heckman measures can be defined for general  Fano fibrations, associated with a pair of filtrations satisfying certain ``properness" conditions. This approach would play a role in studying the convexity of various functionals and the uniqueness of the weighted volume minimizer. 
       \end{remark}
\subsection{The case when $X$ is smooth}
The above definition of weighted volume is motivated to give an algebro-geometric interpretation of the weighted volume in \cite{ConlonDS} which was defined in terms of symplectic geometry. 
Suppose $(\pi: X\rightarrow Y,\xi)$ is a polarized Fano fibration with $X$ smooth. We now explain the two definitions of weighted volume coincide.

First, using the fact that the polarized Fano fibration can be equivariantly embedded into $\C^{N_1}\times \C\mathbb P^{N_2}$, it is easy to construct a 
$\mathbb T$-invariant positively curved smooth hermitian metric on $-K_X$ with curvature form $-2\pi\sqrt{-1}\omega$ (then $\omega\in c_1(X)$), such that the $\mathbb T$-action is Hamiltonian with respect to $\omega$ with a moment map $\mu: X\rightarrow \mathrm{Lie}(\mathbb T)^*$ and $u_\xi:=\langle\mu, \xi\rangle$ is proper and bounded below on $X$. 
    Note that moment maps are not unique and they differ by an element in $\mathrm{Lie}(\mathbb T)^*$: a choice of the moment map corresponds to a lift of the $\mathbb T$-action to the line bundle $-K_X$. Here we have a canonical lifting for the $\mathbb T$-action, as automorphisms of the manifold can push forward tangent vectors. In the following, we work with the moment map associated with this canonical lifting and note that the properness of $u_{\xi}$ is independent of the choice of the moment map.

In \cite{ConlonDS}, the \textit{analytic} weighted volum is defined  as 
\begin{equation}\label{analytic volume}
    {\Bbb W}^{an}(\eta):=\frac{1}{(2\pi)^n}\int_X e^{-\langle\mu, \eta\rangle}\frac{\omega^n}{n!}.
\end{equation}
Using the Duistermaat-Heckman formula  \cite{PW}, it is shown in \cite{ConlonDS} that \eqref{analytic volume} is well-defined for $\eta$ belonging to an open convex cone $\Lambda\subset \mathrm{Lie}(\mathbb T)$, which contains $\xi$.

\begin{proposition}\label{prop--analytic volume coincidew with algebraic}
	In the above setting we have $${\Bbb W}^{an}(\xi)={\Bbb W}(\xi),$$ where ${\Bbb W}(\xi)$ is defined in \eqref{e:weighted volume for vector fields}. 
\end{proposition}

\begin{proof}
	Since $\pi:X\rightarrow Y$ is a polarized Fano fibration, $X$ is weakly pseudo-convex, and $-K_X$ admits a smooth hermitian metric with positive curvature form. Then by \cite[Chapter VIII-Theorem 5.6]{Demailly}, we know that for all $k\geq 1$ and $m\geq 0$
 $$H^k(X, -mK_X)=0.$$  Then we can apply \cite[Proposition 7.2]{ConlonDS}, which uses Wu's Riemann-Roch type formula \cite{Wu03},  to obtain that 
    \begin{equation}\label{formal character}
    	{\Bbb W}^{an}(\xi)=\lim_{m\rightarrow\infty}\frac{1}{m^n}\text{char} H^0(X, -mK_X)\left(\frac{\xi}{m}\right),
    \end{equation}
    where $\text{char}$ denotes the (formal) character for the $\mathbb T$-action. Here by the choice of the moment map, the right-hand side of \eqref{formal character} equals
    \begin{equation}
\lim_{m\rightarrow\infty}\frac{1}{m^n}\sum_{\alpha \in \mathrm{Lie}(\mathbb T)^*} \dim_{\mathbb C}(R_{m,\alpha})e^{-\langle\frac{\alpha}{m},\xi\rangle}.
    \end{equation}
By Proposition \ref{lem-convergence of measures} and Lemma \ref{p:two valuations},  the measure
\begin{equation}
    \frac{1}{m^n}\sum_{\alpha\in \mathrm{Lie}(\mathbb T)^*}\dim_{\mathbb C}(R_{m,\alpha})\delta_{\langle\frac{\alpha}{m},\xi\rangle}
\end{equation}converges weakly to $\mathrm{DH}(\xi).$ Therefore, by the definition of ${\Bbb W}(\xi)$, it is sufficient to verify the validity of interchange of limits and integration. Integration by parts, we obtain that 
\begin{equation}\label{intergration by part trick}
        \frac{1}{m^n}\sum_{\langle\frac{\alpha}{m},\xi\rangle\geq A} \dim_{\mathbb C}(R_{m,\alpha})e^{-\langle\frac{\alpha}{m},\xi\rangle}\leq \frac{1}{m^n}\int_A^{\infty}e^{-t}\dim_{\mathbb C}(R_m/ \mathcal{F}^{>tm}_{\mathrm{wt}_{\xi}}R_m) d t.
\end{equation}Then the conclusion follows from \eqref{intergration by part trick} and the uniform dimension estimate \eqref{estimate for the length}.
\end{proof}

    As in \cite{ConlonDS}, there is an explicit expression for the moment map corresponding to the canonical lifting. By construction there is a $\mathbb T$-invariant Ricci potential $h$ for $\omega$, i.e,
    \begin{equation}
        \mathrm{Ric}(\omega)+\ii\p\pp h=\omega.
    \end{equation}
(Note that it is proved in \cite[Lemma 5.10]{ConlonDS} that if $(\omega,f)$ is a K\"ahler-Ricci shrinker, the moment map always exists and $f$ is automatically invariant under the $\mathbb T$-action. So in this case we can choose $h=f$.)

Then we can normalize the moment map for $\eta\in \mathrm{Lie}(\mathbb T)$ as follows:
    \begin{equation}\label{eq--normalization}
        \Delta_{\omega}u_{\eta}+u_{\eta}+\frac{1}{2}J\eta\cdot h=0.
    \end{equation}We can easily check that this normalization corresponds to the canonical lifting of the $\mathbb T$-action. Indeed, note that
        \begin{equation}\label{divergence of vector fields}
            \Delta_{\omega}u_{\eta}=-\frac{1}{2}\mathrm{div}(J\eta)=\frac{1}{2}\frac{\mathcal L_{-J \eta}\omega^n}{\omega^n}.
        \end{equation}Therefore at a zero $x$ of $\eta$,  by \eqref{eq--normalization} and \eqref{divergence of vector fields}, for $s\in (-K_X)_x$ we have
        \begin{equation}
            \mathcal L_{\eta}s=\ii u_{\eta}s.
        \end{equation}This implies that  $u_{\eta}$ is the moment map corresponding to the canonical lifting of the $\mathbb T$-action on $-K_X$.

The weighted volume of a shrinker is closely related to Perelman's ${\mathcal W}$-entropy, $\mu$-functional and reduced volume \cite{Perelman2002}. Recall that on a compact real $m$-dimensional Riemannian manifold $(X,g)$ with a smooth function $\phi$ and $\tau>0$, Perelman's ${\mathcal  W}$-entropy and $\mu$-functional are defined as 
\begin{equation}
    \mathcal{W}(g, \phi, \tau)=\int_X\left(\tau\left(|\nabla \phi|^2+R\right)+\phi-n\right) \frac{e^{-\phi}}{(4 \pi \tau)^{m / 2}} d V_g,
\end{equation}
\begin{equation}
    \mu(g,\tau)=\inf\{\mathcal{W}(g,\phi,\tau)\mid \int_X\frac{e^{-\phi}}{(4 \pi \tau)^{m / 2}} d V_g=1\}.
\end{equation}This definition can be generalized to non-compact Ricci shrinkers \cite{LiWang}. 
In this paper, we work with K\"ahler-Ricci shrinkers $(g,J,\omega,f)$ under the convention  \eqref{e:1.1} and normalize $f$ such that \eqref{normalization of the soliton potential} holds, that is we have
\begin{equation}
\Delta_{\omega}f+f-|\nabla^{1,0} f|^2=0.
\end{equation}The results from \cite{LiWang,CN09} show that 
 the function $\mu(g,\tau)$ is decreasing on $(0,\frac{1}{2})$ and increasing on $(\frac{1}{2},\infty)$ and moreover, 
\begin{equation}
     \frac{1}{(2\pi)^n}\int_{X}e^{-(f+n)}\frac{\omega^n}{n!}=e^{\mu(g,\frac{1}{2})}.
\end{equation}
Let $\xi$ denote the soliton vector field as before. Then by Proposition \ref{prop--analytic volume coincidew with algebraic} and \eqref{analytic volume}, we obtain that on a K\"ahler-Ricci shrinker, the algebraically defined weighted volume satisfies
\begin{equation}\label{relation between weighted volume and mu functional}
    \mathbb{W}(\xi)=e^{\mu(g,\frac{1}{2})+n}.
\end{equation}

\section{The conjectures}\label{section--conjectures}

\subsection{YTD type conjecture}\label{sec--ytd type}
Given a polarized Fano fibration $(\pi:X\rightarrow Y, \xi)$ with $X$ smooth, we say it admits a K\"ahler-Ricci shrinker if there is a K\"ahler-Ricci shrinker $(\omega, f)$ on $X$ such that $J\nabla f=\xi$. 
Now we  make the following conjecture
\begin{conjecture}\label{YTD type conj}
	A polarized Fano fibration $(\pi: X\rightarrow Y, \xi)$ with $X$ smooth admits a K\"ahler-Ricci shrinker, which is unique up to the action of $\mathrm{Aut}(X, \xi)$,  if and only if it is K-polystable. 
\end{conjecture} 

The extension to the case when  $X$ has klt singularities requires a little explanation. One can use pluripotential theory to make sense of a (singular) K\"ahler-Ricci shrinker as in \cite{BermanW}. This means that $\omega$ is a $\mathbb T$-invariant positive $(1,1)$ current given by the curvature of  a singular hermitian metric $h$ on $K_{X}^{-r}$ (where $r$ is the Cartier index of $X$) which has locally bounded potential and is smooth on the smooth locus of $X$,  $\xi$ has a Hamiltonian function $f$ (with respect to $\omega$) which is locally bounded and globally bounded from below and proper, such that the measure $e^{-f}\omega^n$ has finite total mass and the corresponding complex Monge-Amp\`ere equation holds
$$e^{-f}\omega^n=\Omega_h,$$
where $\Omega_h$ is the volume form on $X$ defined by $h$. 
 In the special setting when $Y$ is a point this reduces to a K\"ahler-Ricci shrinker on a polarized Fano variety. In the special setting when $\pi$ is the identity map one expects that this reduces to  a  Ricci-flat K\"ahler cone metric on a Fano cone (see Conjecture \ref{conj2}). When $X$ is smooth, it is not a priori clear whether a weak solution as above automatically defines a complete metric. 
 Using the above notion of K\"ahler-Ricci shrinkers, Conjecture \ref{c:1-1} is the counterpart of Conjecture \ref{YTD type conj} in the singular setting.

It is also an open question to understand the metric geometry of a  K\"ahler-Ricci shrinker. For example, one expects suitable conical behavior at infinity.
\begin{conjecture}[Asymptotic conicity]\label{C:asymp conicity} 
A K\"ahler-Ricci shrinker is weakly asymptotically conical in the sense that it has a unique asymptotic cone at infinity, which is a  K\"ahler cone metric on $Y$. 
\end{conjecture}
 The interest of these conjectures lies in the analysis in the non-compact setting, \emph{without} a priori prescribing the behavior at infinity. This is in the spirit of a free boundary problem.  Some explicit examples of K\"ahler-Ricci shrinkers using the Calabi ansatz  have been constructed in \cite{Koiso, Cao1996, FIK2003, DW2011, FW2011, C.Li2010}. 
 
 The uniqueness part of Conjecture \ref{YTD type conj} was proved in \cite{Cifarelli} in the toric setting. Under the assumption of quadratic curvature decay, Conjecture \ref{C:asymp conicity} holds by \cite{MW2017, ConlonDS}, and there is work by Esparza \cite{Esparza} on the uniqueness part of Conjecture \ref{YTD type conj}.  
In the toric setting, one expects that a toric Fano fibration always admits a unique K-polystable polarization and hence admits a toric K\"ahler-Ricci shrinker. This is known in the compact setting \cite{WZ04, BB2013} and progress has been made by \cite{CCD2} in the non-compact setting.

    \

When $\pi$ is the identity map, by definition a  Ricci-flat K\"ahler cone metric is a  K\"ahler-Ricci shrinker. On the other hand, the K-stability notions for these two objects coincide on a Fano cone. The uniqueness part in Conjecture \ref{c:1-1} would imply the converse:

\begin{conjecture}\label{conj2}
On a Fano cone $X$, every K\"ahler-Ricci shrinker is a  Ricci-flat K\"ahler cone metric.
\end{conjecture}

This is a purely differential geometric statement. We give heuristic arguments here under extra assumptions. Let $g$ denote the Riemannian metric of a K\"ahler-Ricci shrinker on $X$ with soliton vector field $\xi$. Suppose one can apply the results in \cite{DS2} to get a 2-step degeneration to its unique tangent cone $(\ca,g_\ca)$ at $p$, which is a Ricci-flat K\"ahler cone. Then by the monotonicity \cite{LW2020} and upper-semicontinuity \cite[Lemma 6.28]{BChow} of the $\mu$-functional, we can get 
\begin{equation}\label{properties for mu functional}
    \mu(g_\ca,\frac{1}{2})\geq \limsup_{\tau\rightarrow 0}\mu(\tau^{-1}g,\frac{1}{2})=\limsup_{\tau\rightarrow 0}\mu(g,\tau)\geq \mu(g,\frac{1}{2}).
\end{equation} Then using \cite{CN09} (see also \cite[Proposition 3]{LW2020}), one sees that (see \eqref{relation between weighted volume and mu functional} for the smooth case)
\begin{equation}
    e^{\mu(g_{\ca},\frac{1}{2})+n}={\Bbb W}(v_g) \quad \text{and}\quad e^{\mu(g,\frac{1}{2})+n}={\Bbb W}(\xi),
\end{equation}where $v_g$ denotes the valuation induced by the metric $g$ \cite[Section 3.2]{DS2}. Since $g_\ca$ is a Ricci-flat K\"ahler cone metric, the valuation $v_g$ minimizes the weighted volume \cite{LX2018}, and then the equality holds in \eqref{properties for mu functional}. From the variation of the $\mu$-functional, we can see that $g$ is a cone metric with Reeb vector field $\xi$, implying $g$ must be Ricci-flat since its tangent cone at $p$ is Ricci-flat.

\subsection{Weighted volume minimization}
We propose the following 2-step degeneration conjecture for weighted volume minimization. The goal is to degenerate a Fano fibration germ to a K-polystable polarized Fano fibration.

\begin{conjecture}(2-step degeneration for weighted volume)\label{weighted volume mini conj}
Given a Fano fibration germ $(\pi: X\rightarrow Y, p)$. 
	\begin{itemize}
		\item [(1).] There is a unique minimizer  $v_*$ of ${\Bbb W}$ on $\Val_{\pi, p}^*$, which is quasi-monomial. 
  \item [(2).] The graded ring  $\operatorname{Gr}_{v_*}(\oo_{Y,p})$ is  finitely generated  and $\bigoplus_{m\geq 0}\operatorname{Gr}_{v_*}(R_m)$ is a finitely generated algebra over $\operatorname{Gr}_{v_*}(\oo_{Y,p})$.
		\item [(3).] Let $W=\operatorname{Spec}(\operatorname{Gr}_{v_*}(\oo_{Y,p}))$ and $Z=\operatorname{Proj}_{W}(\bigoplus_{m\geq 0}\operatorname{Gr}_{v_*}(R_m))$. Then $v_*$ induces a vector field $\xi$ on $Z$ which descends to $W$. The natural map $\tilde \pi:Z\rightarrow W$ and $\xi$ defines a K-semistable polarized Fano fibration. Moreover, $v_*$ is uniquely characterized by this K-semistability property.
		\item[(4).] There is an equivariant test configuration degenerating $(\tilde \pi:Z\rightarrow W, \xi)$ to a unique K-polystable polarized Fano fibration $(\widehat\pi: \mathcal S\rightarrow\mathcal C, \xi)$.
	\end{itemize}
\end{conjecture}

This brings together a few special situations studied previously. When $\pi$ is the identity map this essentially reduces to the minimization of normalized volume for klt singularities in  \cite{Li2015}. 
When $Y$ is a point, this reduces to the 2-step degeneration theory for Fano varieties \cite{DeSz, HL2024}. In these two cases the conjecture is known by the work of many people; see \cite{XZ, BLLZ2023} and the references therein.  One expects that suitable extensions of these techniques can be used to prove Conjecture \ref{weighted volume mini conj}.

\begin{definition} \label{minimal weighted volume} We define the weighted volume of
\

\begin{itemize}
	\item    a Fano fibration germ $(\pi: X\rightarrow Y, p)$ to be ${\Bbb W}(\pi, p):=\inf_{v\in \Val_{\pi, p}^*}\mathbb W(v)$;
 \item a Fano variety $X$ to be 
 ${\Bbb W}(X):=\mathbb W(\pi: X\rightarrow \{pt\}, \{pt\})$;
 \item   a klt singularity $(X, q)$ to be  ${\Bbb W}(X, q):=\mathbb W(\Id: X\rightarrow X, q)$.
 \end{itemize}
\end{definition}

Given a Fano fibration germ $(\pi: X\rightarrow Y, p)$ and $q\in \pi^{-1}(p)$, by definition we have the following local-to-global comparison for weighted volume
    \begin{equation}\label{e:weighted volume comparison}
        {\Bbb W}(X, q)\geq {\Bbb W}(\pi, p). 
    \end{equation}
     This is the algebro-geometric analogue of the monotonicity of the $\mu$-functional \cite{LiWang} for K\"ahler-Ricci shrinkers discussed in Section \ref{sec--ytd type} above. For K\"ahler-Einstein metrics, the algebro-geometric analogue of the Bishop-Gromov volume comparison was established by \cite{Liu2018}.
   The differential geometric heuristics suggest that the equality in \eqref{e:weighted volume comparison} holds if and only if $\pi$ is the identity map.

It is also interesting to investigate more closely the set of weighted volumes of Fano fibration germs.
As usual one expects that the weighted volume is lower semi-continuous in flat families of Fano fibration germs, which would imply the maximal volume of a birational Fano fibration germ in dimension $n$ is uniquely achieved by the standard $(\Id: \C^n\rightarrow \C^n, 0)$. Compare \cite{Fujita, LiuXu, BlumLiu} for results on similar questions in related settings. 
Here, a Fano fibration germ $(\pi: X\rightarrow Y, p)$  is called \emph{birational} if $\pi$ is birational. 
 % For birational Fano fibrations, there is a well-known FIK K\"ahler-Ricci shrinker \cite{FIK2003} on  $\pi: \mathrm{Bl}_0(\C^n)\rightarrow \C^n$, whose Reeb vector field is a multiple of the standard radial vector field on $\C^n$ and whose weighted volume has an explicit formula (see Appendix A.8 in \cite{ConlonDS}). We refer to this as the \emph{FIK Fano fibration}.
Moreover, as an analogue of the ODP volume gap conjecture (see \cite{SS}), it is natural to ask:
\begin{question}
Determine the second largest weighted volume of a birational Fano fibration germ in dimension $n$.
\end{question}
There are two natural candidates:  $\pi: \mathrm{Bl}_0(\C^n)\rightarrow \C^n$ and the ODP klt singularity (with $\pi=\Id$). Note that there is a well-known FIK K\"ahler-Ricci shrinker \cite{FIK2003} on  $\pi:\mathrm{Bl}_0(\C^n)\rightarrow \C^n$, whose Reeb vector field is a multiple of the standard rotational vector field on $\C^n$ and whose weighted volume has an explicit formula; see \cite[Appendix A.4]{ConlonDS}.

\subsection{Moduli theory}\label{sec:Moduli space}
The study of canonical metrics in K\"ahler geometry is closely related to the construction of moduli spaces in algebraic geometry. One prominent example is the case of K\"ahler-Einstein metrics on Fano varieties. Building on the differential geometric results of \cite{DS1, CDS, SSY}, the compactified moduli space of smooth K-polystable Fano varieties was constructed by  \cite{LWX1} and \cite{Odaka}. Later this has been generalized to non-smoothable Fano varieties in algebraic geometry by \cite{CJiang, BLX2022, XuZhuang2020positivity,LXZ}. There are also related study of moduli theory for Fano cone singularities (corresponding to Ricci-flat K\"ahler cone metrics) and for polarized Fano varieties (corresponding to compact K\"ahler-Ricci shrinkers) in both differential and algebraic geometry; see for example \cite{Sun23, HLW, Xu-Zhuang2024,  Y.Odaka2024, Odaka0, Chen2024}. We expect all of these can be incorporated into one theory:

\begin{conjecture}(K-moduli for polarized Fano fibrations)
   Given a positive integer $n$ and $\epsilon>0$, there exists a projective moduli space parametrizing K-polystable polarized Fano fibrations of dimension $n$ and weighted volume at least $\epsilon$. 
\end{conjecture}

Here the weighted volume of a K-polystable polarized Fano fibration $(\pi: X\rightarrow Y, \xi)$ is defined to be ${\Bbb W}(\xi)$ (c.f. \eqref{e:weighted volume for vector fields}).  In particular, we expect:
    \begin{itemize}
\item (Boundedness) The set of K-semistable polarized Fano fibrations with weighted volume bounded below by a positive constant $\epsilon>0$ forms a bounded family.  
    \end{itemize}
 This question is related to the weak compactness of Ricci shrinkers (see \cite{HLW,LLW2021} for instance). 
Notice that in the case when $Y$ is a point,  the boundedness question above is known, as a consequence of \eqref{e:weighted volume comparison}; for instance, one way to see this is through 
% \cite[Theorem 6.13]{LiLiuXu} and Birkar's boundedness on $\epsilon$-lc Fano varieties \cite{Birkar2021}. 
 the finite degree formula in \cite{Xu-Zhuang2021} and the boundedness result in \cite{HMX}.

   In birational geometry, a well-known question is on the termination of flips in the minimal model program. This is related to bounding the weighted volume from below for the Fano fibrations coming from the extremal contractions along a fixed minimal model program.

\subsection{K\"ahler-Ricci flow I: finite time singularities}\label{finite time singularities}

Consider the Ricci-flow on a compact K\"ahler manifold $(X, \omega)$ 
\begin{equation}\label{Ricci flow equation}
\frac{\p}{\p t}\omega(t)=-\mathrm{Ric}(\omega(t))
\end{equation}
with $\omega(0)=\omega$. 
In \cite{SongTian}, it is proposed that analyzing this flow constitutes an analytic analogue of the minimal model program (with scaling) in birational geometry, often referred to as the analytic minimal model program. 
It is well-known \cite{TianZhang} that the flow exists exactly on $[0, T)$, where $T$ can be numerically computed as the supremum of all $t>0$ such that $[\omega]-2\pi tc_1(X)$ is a K\"ahler class. 

If $K_X$ is not nef, then $T<\infty$ and the flow must encounter finite time singularities at $T$. From now on we assume 
 $X$ is projective and the initial K\"ahler class $[\omega(0)]=2\pi c_1(L)$ for some line bundle $L$. Then by the Kawamata basepoint free theorem, we obtain a Fano fibration  $\pi: X\rightarrow Y$  induced by the semi-ample $\mathbb Q$-line bundle $L+TK_X$.
 When $X$ is Fano and $[\omega]\in 2\pi c_1(X)$, it is shown in \cite{CSW} that there is a 2-step degeneration from the normalized K\"ahler-Ricci flow to a unique K\"ahler-Ricci shrinker on a $\mathbb Q$-Fano variety.
 Motivated by this we make the following 

\begin{conjecture}[2-step degeneration for K\"ahler-Ricci flows]\label{c:KRF vs weighted volume}
Let $\pi:X\rightarrow Y$ be the Fano fibration at the singular time $T<\infty$.	 For any $p\in Y$ and $q\in \pi^{-1}(p)$,
  \begin{itemize}
      \item  The K\"ahler-Ricci flow $\omega(t)$ naturally induces a weighted volume minimizing (hence K-semistable) valuation on $(\pi: X\rightarrow Y, p)$. 
      \item The tangent flow of $\omega(t)$ at $q$ is unique, and is given by the K\"ahler-Ricci shrinker $(\widehat\pi: \mathcal S\rightarrow\mathcal C, \xi)$ canonically associated with the Fano fibration germ $(\pi: X\rightarrow Y, p)$ via Conjecture \ref{weighted volume mini conj} and Conjecture \ref{c:1-1}.
  \end{itemize}
  \end{conjecture}
  
  \begin{remark}
      As mentioned in the introduction, one can more generally consider a K\"ahler-Ricci flow space-time with singularities in a suitable sense and study the tangent flow at a singularity.  This will then include the above situation as well as the case of tangent cones at singularities of  K\"ahler-Einstein metrics in the sense of \cite{DS2} (which are self-similar Ricci flow solutions).
  \end{remark}
  \begin{remark}
    When the Fano fibration $\pi$ is birational, Conjecture \ref{c:KRF vs weighted volume} gives a precise characterization of the K\"ahler-Ricci shrinker appearing in the shrinker-to-expander transition, conjectured in \cite{Song17} for the analytic minimal model program.  As in \cite{Song17}, one expects there exists a K\"ahler-Ricci expander $\mathcal E$, whose asymptotic cone coincides with that of the K\"ahler-Ricci shrinker produced in this conjecture. The underlying variety of $\mathcal E$ is the unique canonical model for $K_{\mathcal S}$ over $\mathcal C$, whose existence is guaranteed by  \cite{BCHM}. There has been some work constructing K\"ahler-Ricci expanders; see \cite{Conlond-Deruelle} and the references therein.
  \end{remark}
 Theorem \ref{t:main1} is a first step towards proving Conjecture \ref{c:KRF vs weighted volume}. The uniqueness of tangent flows with singularities is in general a difficult question in geometric analysis, and the above conjecture leads to an algebro-geometric way of proving such uniqueness. This strategy has been successfully realized in \cite{DS2} for the uniqueness of tangent cones of singular K\"ahler-Einstein metrics and in \cite{CSW} for the uniqueness of tangent flows for the Ricci flow on Fano manifolds. It is possible to apply the techniques in \cite{DS2, CSW} and the weak compactness theory in \cite{Bamler} to tackle Conjecture \ref{c:KRF vs weighted volume}. 
 Some partial progress towards generalizing the infinitesimal result of \cite{DS2} has been made by \cite{Hallgren2023}. We mention that there are also results explicitly determining the tangent flows in special settings, see \cite{CCD,JST}.

 Regarding the first part of the conjecture, as in \cite{CSW} one can try to write down a formula for the expected valuation induced by the K\"ahler-Ricci flow $\omega(t)$. 
Fix a smooth hermitian metric $H_T$ on the $\mathbb Q$-line bundle $L+TK_X$. Combining this with the hermitian metrics on $L+tK_X$, whose curvature forms are $\omega_t$, we obtain a family of hermitian metrics $h_t$ on $-K_X$ (and thus also on $-mK_X$). We normalize $h_t$ so that its associated volume form has total mass $1$. For a holomorphic section $s$ of $-mK_X$ defined in a neighborhood $U$ of $\pi^{-1}(p)$,  we define the $L^2$-norm 
 $$\|s\|_t^2=\frac{1}{(T-t)^n}\int_U |s|_{h_t}^2 \omega_t^n.$$
 Then we expect the following
 \begin{equation}
\nu(s):=\lim_{t\rightarrow T^-}\frac{\log \|s\|_{t}}{\log(T-t)}
     \end{equation}
     is well-defined and gives rise to a valuation (independent of the choice of $H_T$ and $U$) on the ring $\mathcal{R}$ (see \eqref{e:4-1}),  which induces the K-semistable valuation in $\Val_{\pi, p}^*$.

\subsection{K\"ahler Ricci flow II: ancient solutions}
By drawing an analogy with the theory of complete Ricci-flat K\"ahler metrics with Euclidean volume growth, one can also reverse the rescaling direction and consider non-collapsed ancient solutions to the K\"ahler-Ricci flow. By definition this means a K\"ahler-Ricci flow $g(t) (t\in (-\infty, 0])$ of complete K\"ahler metrics on a fixed complex manifold $X$ with  Perelman's $\mu$-functional $\mu(g(t),-t)$ uniformly bounded from below. Algebro-geometrically we introduce: 

\begin{definition}[Non-collapsing weak Fano fibration]
   A non-collapsing weak Fano fibration is a fibration $\pi: X\rightarrow Y$, where $X$ has klt singularities, $Y$ is a normal affine variety and 
   $-K_X$ is a relatively nef and big $\mathbb Q$-Cartier divisor.
\end{definition}

Notice that the relative bigness of $-K_X$ and the affineness of $Y$ reflect the essence of the non-collapsing condition. Denote $\mathfrak R_m=H^0(X, -mK_X)$. In particular, $\mathfrak R_0$ is the ring of regular functions on $X$ (and on $Y$), which is a finitely generated $\mathbb C$-algebra. Let $\mathfrak R:=\bigoplus_{m\geq 0}\mathfrak R_m$ be the global anti-canonical ring. By the relative base point free theorem \cite[Theorem 3.24]{KM}, $\mathfrak R$ is a finitely generated graded algebra over $\mathfrak R_0$.
A negative valuation on a non-collapsing weak Fano fibration $\pi: X\rightarrow Y$ consists of a function $d: \mathfrak{R}\setminus \{0\}\rightarrow \mathbb R_{\leq 0}$ satisfying 
\begin{itemize}
    \item $d(s_1+s_2)\geq \min\{d(s_1), d(s_2)\}$;
    \item $d(s_1 s_2)=d(s_1)+d(s_2)$.
\end{itemize}

\begin{definition}
    We say a negative valuation is K-polystable if the associated graded algebra is finitely generated and with the reversed grading it defines a K-polystable polarized Fano fibration $(\pi: \mathcal S\rightarrow \mathcal C, \xi)$.
\end{definition}
 We make the following conjecture, which serves as a global version of Conjecture \ref{c:KRF vs weighted volume}.

\begin{conjecture}[Tangent flows at infinity]\label{tangent cone at infinity}
Given a non-collapsed ancient K\"ahler-Ricci flow $(X, \omega(t))(t\in (-\infty, 0])$. Then
\begin{itemize}
\item $X$ admits a non-collapsing weak Fano fibration structure and the flow naturally induces a K-polystable negative valuation;
    \item The tangent flow at infinity of $\omega(t)$ is unique and is given by the K\"ahler-Ricci shrinker determined by the K-polystable negative valuation via Conjecture \ref{c:1-1}. 
\end{itemize}
\end{conjecture}
Similar to the discussion in Section \ref{finite time singularities}, we can write down a formula for the negative valuation determined by the ancient flow. For simplicity, we assume that $t^{-1}(\omega(0)-\omega(t))$ is given by the curvature form of a hermitian metric $h_t$ on $-K_X$, and we assume $h_t$ can be normalized so that the total mass of the associated volume form on $X$ is $1$. Then  on the (holomorphic) anti-canonical section ring of $X$, we can consider:
\begin{equation}
    \|s\|_t^2:=\frac{1}{(-t)^n}\int_X |s|_{h_t}^2 \omega_t^n \text{ and } d(s):=\lim_{t\rightarrow -\infty} \frac{\log \|s\|^2_{t}}{\log(-t)}.
    \end{equation}
    We expect that $d$ is well-defined on the subspace consisting of sections $s$ with $\|s\|_t<\infty$,  and gives a K-polystable negative valuation.

Note that a ``no-semistability at infinity" phenomenon is expected in Conjecture \ref{tangent cone at infinity}, which is motivated by \cite{SunZ}. Moreover,  conjecturally non-collapsed ancient K\"ahler-Ricci flows can be characterized by the following algebraic data: 

\begin{itemize}
\item A choice of a K-polystable polarized Fano fibration $(\pi: \mathcal S\rightarrow \mathcal C, \xi)$.
    \item A choice of a K-polytable negative valuation on a non-collapsing weak Fano fibration $(\pi:X\rightarrow Y)$, which degenerates it to $(\pi: \mathcal S\rightarrow \mathcal C, \xi)$.
    \item A choice of a K\"ahler class on $X$ (corresponding to $[\omega(0)]$).
\end{itemize}

\begin{remark}
    The weighted volume of the K\"ahler-Ricci shrinkers obtained in Conjecture \ref{c:KRF vs weighted volume} and \ref{tangent cone at infinity} are related to Perelman's ${\mathcal W}$-entropy along the Ricci flow; see \cite{MM2015,GXu,MaZhang,Hallgren}.
\end{remark}

\appendix 

\section{Proof of Proposition \ref{lem-convergence of measures}}
 If $Y$ is a point, this has been proved in \cite{BoucChen2011}, so we only consider the case when $\dim Y>0$.
Note that 
\begin{equation}
    \mathrm{DH}_m=\frac{1}{(rm)^n}\frac{d}{d t}\dim_{\mathbb C}(R_{m}/\mathcal{F}_v^{trm} R_{m}).
\end{equation}In the following, for simplicity of notations, we assume $r=1$ and it follows from the proof that the limit measure does not change if we replace $L$ with a multiple.
    By the compactness of Radon measures and the uniform estimates \eqref{estimate for the length}, we know that $\mathrm {DH}_m$ is weakly precompact. As in \cite{BoucChen2011}, by the dominated convergence theorem, it suffices to show that the limit 
    \begin{equation}
        \lim_{m\rightarrow \infty}\frac{1}{m^n} \dim_{\mathbb C}(R_{m}/\mathcal{F}_v^{tm} R_{m})
    \end{equation}exists for  almost every $t\geq0$. Indeed, in the case $\dim Y>0$, we can show that the above limit exists for every $t\geq0$.

The strategy is to apply the following theorem  due to  \cite{KK08} and  \cite{LM2009}.   
\begin{theorem}[\cite{LM2009}]\label{okounkov body} Denote by $\mathbb N$ the set of non-negative integers and let $\Gamma \subset \mathbb N^{n+1}$ be a sub-semigroup. For $m\in \mathbb N$, we denote
$\Gamma_m=\Gamma\cap(\mathbb N^n\times \{m\}),
   $ which can be viewed as a subset of $\mathbb N^n$. Suppose $\Gamma$ satisfies the following three conditions: 
   \begin{itemize}
    \item [(1).] $\Gamma_0=\{0\}$;
    \item [(2).] There exist finitely many vectors $(v_i,1)$ spanning a semigroup $B\subset \mathbb N^{n+1}$ such that $\Gamma \subset B$;
    \item [(3).] $\Gamma$ generates $\mathbb Z^{n+1}$ as a group.
   \end{itemize}Then the limit
   $$\lim_{m\rightarrow \infty}\frac{\#\Gamma_m}{m^n}$$
exists and equals the Euclidean volume of the Okounkov body of the semigroup $\Gamma$.
 \end{theorem}

Take a log resolution $\varphi:\widetilde X\rightarrow X$, then the center of $v$ on $\widetilde X$ satisfies that  $\overline{\{c_{\widetilde X}(v)\}}\subset \varphi^{-1}(\pi^{-1}(p))$. Choose a point $\widetilde q\in \overline{\{c_{\widetilde X}(v)\}}$ such that $c_{\widetilde X}(v)$ is smooth in a neighborhood of $\tilde q$. In an affine chart around $\widetilde q$ we assume $c_{\widetilde X}(v)$ is cut out by $k$ regular functions $z_1, \cdots, z_k$. Then $v(z_i)>0$ for all $i=1,\cdots,k$.  We may add regular functions $z_{k+1}, \cdots, z_n$ such that $\{z_1, \cdots, z_n\}$ form local analytic coordinates around $\widetilde q$.

Choose  rationally independent positive real numbers $\alpha_1,\cdots,\alpha_n$, with 
\begin{equation}\label{upper bound for alpha_i}
    \alpha_i\leq \min \{v(z_i), 1\}, i=1\cdots,k.
\end{equation} Consider two quasi-monomial valuations $\mathrm{val}$ and $\widetilde v$ given by 
 \begin{equation}
     \mathrm{val}(f)=\min\{\sum^n_{i=1} \alpha_i\beta_i\mid c_{\beta}\neq 0\}. \quad \tilde v(f)=\min\{\sum^k_{i=1} \alpha_i\beta_i\mid c_{\beta}\neq 0\},
 \end{equation}
 for $0\neq f\in \oo_{\widetilde X,\widetilde q}$,  written in the completion  $\widehat{\oo_{\widetilde X,\widetilde q}}\simeq \mathbb C[[ z_1, \ldots, z_n ]]$ as $f=\sum_{\beta} c_\beta z^\beta$ with $c_{\beta}\in \mathbb C$. 
By the choice of $\alpha_i$ \eqref{upper bound for alpha_i} and \cite[Proposition 3.1]{JM2012}, we know that for $f\in \oo_{\widetilde X,c_{\widetilde X}(v)}$,
\begin{equation}\label{bounded below by quasimonomial}
    v(f)\geq \tilde v(f).
\end{equation}
Since the $\alpha_i$'s are rationally independent, for $f\in \oo_{\mathrm{val}}\setminus \{0\}$, we can define 
\begin{equation}\label{definition of vec valuation}
    \vv{\mathrm{val}}(f):=(\mathrm{val}_1(f),\cdots, \mathrm{val}_n(f))\in \mathbb N^n
\end{equation}
 if $ \mathrm{val}(f)=\mathrm{val}_1(f)\alpha_1+\cdots \mathrm{val}_n(f)\alpha_n$. Then we have the following result.
\begin{lemma}
    There exists a positive constant $C$ such that for any $m\geq 0$ and $0\neq s\in R_m$, we have 
    \begin{equation}\label{eq--upper bound of val}
          \sum_{i=1}^n\mathrm{val}_i(s)\leq C(\widetilde v(s)+m).
    \end{equation}
\end{lemma}
\begin{proof}
    We consider  the blowup $\varphi': X'\rightarrow \widetilde X$ of $\widetilde X$ along the subvariety $\overline{\{c_{\widetilde X}(v)\}}$ and let $E$ denote the exceptional divisor.
Then there exists a positive constant $C$ such that for any $0\neq s\in R_m$, there exists a non-negative integer $c(s)$ satisfying \begin{equation}\label{upper bound for vanishing order}
    c(s)\leq C \widetilde v(s),
\end{equation} and $s$ induces a non-zero section $s'$ in 
\begin{equation}
    H^0(E,mL-c(s)E).
\end{equation}
Note that the valuation $\mathrm{val}$ also naturally induces a valuation $\mathrm{val}'$ on $E$. Without loss of generality, we may assume $\mathrm{val}(z_1)=\min\{\mathrm{val}(z_i):i=1,\cdots,k\}$. Let $q'\in X'$ denote the point, around which we have local analytic coordinates: $\{z_1,u_i=\frac{z_i}{z_1},  z_j\mid i=2,\cdots,k,j=k+1,\cdots,n\}$, where $E$ is locally given by $\{z_1=0\}$. Then we define $\mathrm{val}'$ to be the quasi-monomial valuation on $E$ centered at $q'$ with the following weights:
\begin{equation}\label{induced valuation}
    \mathrm{val}'(u_i)=\alpha_i-\alpha_1, i=1,\cdots,k, \quad \mathrm{val}'(z_j)=\alpha_j, j=k+1,\cdots,n.
\end{equation}
Then since $\alpha_1\leq 1$, we can easily check that 
\begin{equation}\label{decomposition of valuations}
    \sum_{i=1}^n\mathrm{val}_i(s)\leq \frac{\mathrm{val}(s)}{\min_{i=1}^n  \alpha_i}\leq  \frac{c(s)+\mathrm{val}'(s')}{\min_{i=1}^n  \alpha_i}.
\end{equation}Fix a very ample line bundle $A$ such that both $L+A$ and $A-E$ are very ample, then we have injective maps:\begin{equation}
    H^0(E,mL-c(s)E)\hookrightarrow H^0(E, (m+c(s))(L+2A-E)).
\end{equation}By definition, $\mathrm{val}'$ is a quasi-monomial valuation on $E$, then by \cite[Proposition 2.12]{BKMS},  there is a positive constant $C$ independent of $s\in R_m$ and $m$ such that
\begin{equation}\label{control on exceptional divisor}
\mathrm{val}'(s')\leq C(c(s)+m).
\end{equation}Then \eqref{eq--upper bound of val} follows from \eqref{upper bound for vanishing order}, \eqref{decomposition of valuations},  and \eqref{control on exceptional divisor}.
\end{proof}

Since the valuation $\mathrm{val}$ is centered at a closed point and the weights $\alpha_i$ are rationally independent, from \eqref{bounded below by quasimonomial} and \eqref{eq--upper bound of val}, we obtain that
 there exists a positive number $N$, independent of $m$ and $t$ such that \begin{equation}\label{conunting lattice points}
 \begin{aligned}
           \dim_{\mathbb C}(R_m/ \mathcal{F}_v^{tm}R_m)=&\#\{(d_1,\cdots, d_n)\in \vv{\mathrm{val}}(R_m)\mid \sum_{i=1}^n d_i\leq N(t+1)m\}\\
           &-\#\{(d_1,\cdots, d_n)\in \vv{\mathrm{val}}(\mathcal F_v^{tm}R_m)\mid \sum_{i=1}^n d_i\leq N(t+1)m\}.\\
 \end{aligned}
 \end{equation}
 Then we have two sub-semigroups of $\mathbb N^{n+1}$ given by 
 \begin{equation}
     \begin{aligned}
         &\Gamma=\{(d_1,\cdots, d_n,m)\mid (d_1,\cdots d_n)\in \vv{\mathrm{val}}(R_m), \sum_{i=1}^n d_i\leq N(t+1)m\},\\
         &\Gamma'=\{(d_1,\cdots, d_n,m)\mid (d_1,\cdots d_n)\in \vv{\mathrm{val}}(\mathcal F_v^{tm}R_m), \sum_{i=1}^n d_i\leq N(t+1)m\}.
     \end{aligned}
 \end{equation}
It is then clear that $\Gamma$ and $\Gamma'$ satisfy Condition (1) and Condition (2), since 
\begin{equation}\label{verify second condition}
\Gamma\subset \{(d_1,\cdots, d_n,m)\mid  \sum_{i=1}^n d_i\leq N(t+1)m\}.
\end{equation}
Therefore we have a uniform dimension estimate for $m\geq 1 $,
 \begin{equation}\label{estimate for the length}
	\dim_{\mathbb C}(R_m/ \mathcal F^{>\lambda m}_vR_{m})\leq C'm^n(\lambda+1)^n,
	\end{equation}
which implies that $\mathrm{DH}_m$ is well-defined.

Since $\Gamma'\subset \Gamma$, to check Condition (3), it suffices to show it holds for $\Gamma'$. We may also allow $N$ to increase if necessary.
 As $L$ is relatively ample, we can find  $m_0$ such that for all $m\geq m_0$, there exists an $f_m\in R_m$ with 
\begin{equation}\label{eq--holomorphic with zero valuation}
     \mathrm{val}(f_m)=0.
 \end{equation} Note that $v$ induces a valuation on $\mathcal O_{Y,p}$ as there is a natural inclusion $\mathcal O_{Y,p}\hookrightarrow K(X)$ via $\pi^*$. Since $\dim Y>0$ and $v$ is centered in $\pi^{-1}(p)$, we can choose a non-zero element $h\in \mathfrak m_p\subset \oo_{Y,p}$ such that 
 \begin{equation}
     v(h)\geq t+1.
 \end{equation}Note that by the choice of $h$, we have $h^mR_m\subset \mathcal{F}_v^{tm}R_m $ and for $m\geq t$, we have 
 \begin{equation}\label{eq--inclusion to high vanihsing}
     \quad h^mR_{m+1}\subset \mathcal{F}_v^{t(m+1)}R_{m+1}. 
 \end{equation}
 Therefore by choosing $m$ large, we get 
 \begin{equation}
     (\vv{0},1)=(\vv{\mathrm{val}}(h^mf_{m+1}), m+1)-(\vv{\mathrm{val}}(h^mf_{m}), m)
 \end{equation}is in the group generated by $\Gamma'$ (with a possibly larger $N$). Again, since $L$ is relatively ample, one can directly check that for each $z_i$, there are  $g_i, s_i\in R_{m_i}$  such that 
 \begin{equation}
     \mathrm{val}(z_i)=\mathrm{val}(g_i)-\mathrm{val}(s_i)
 \end{equation}and then we consider 
 \begin{equation}
     g_i'=h^{m_i}g_i, \text{ and } s_i'=h^{m_i}s_i,
 \end{equation}which are both contained in $\mathcal F_v^{tm_i}R_{m_i}$. Then we know that 
 \begin{equation}\label{eq--generate as a group}
     (e_i, 0)=(\vv{\mathrm{val}}(g_i'), m_i)-(\vv{\mathrm{val}}(s_i'),m_i)
 \end{equation}is in the group generated by $\Gamma'$ (with a possibly larger $N$), where $\{e_i\}_{i=1}^n$ is the standard basis of $\mathbb Z^{n}$. This shows that Condition (3) holds so one can apply Theorem \ref{okounkov body} to get the conclusion.

\section{Proof of Lemma \ref{lem--futaki is well-defined}}
When $\dim Y=0$ this is proved by \cite{BermanW} and it also follows from a more general result in  \cite{BLLZ2023} using Okounkov bodies.  In the following, we assume $\dim Y>0$ and we use Okounkov bodies as in \cite{BLLZ2023}.
We only deal with \eqref{eq--volume is differentiable} and a similar argument works for \eqref{eq-second order derivative}.
Let $\Lambda_{m}=\{\alpha \in\mathrm{Lie}(\mathbb T)^*\mid R_{m,\alpha}\neq 0\}$. In particular, $\Lambda_{0}$ is the weight lattice of the $\mathbb T$-action on $Y$.

Since $-K_X$ is relatively ample and $\xi$ is a polarization, there exist positive constants $A, B$ independent of $m\geq 0$, such that for $\alpha\in \Lambda_m$, we have
	\begin{equation}\label{eq--support of the measure}
	\langle\alpha,\xi\rangle\geq -A m, \text{ and }\left|\langle\alpha,\eta\rangle\right|\leq A \langle\alpha,\xi\rangle+Bm.
	\end{equation}Moreover, similar to \eqref{estimate for the length}, we have the following uniform dimension upper bound for $\lambda\geq 0$ and $m\geq 1$: 
\begin{equation}\label{dimension estimate}
     \sum_{\langle\alpha, \xi\rangle \leq \lambda m}\dim_{\mathbb C}(R_{m,\alpha})\leq C(\lambda+1)^nm^n.
 \end{equation}
We have a sequence of Radon measures on $\mathbb R^2$ given by\begin{equation}
\mathrm{DH}_m(\xi,\eta):=\frac{1}{m^n}\sum_{\alpha \in \Lambda_m}\dim_{\mathbb C}(R_{m,\alpha})\delta_{\left(\langle\frac{\alpha}{m},\xi\rangle,\langle\frac{\alpha}{m},\eta\rangle\right)}.
\end{equation}
% 	Let $N>\max\{A, B\}$ be a large positive constant to be determined later. For $x\in \mathbb R$ we denote
% 	\begin{equation}
% R_m(N, x)=\bigoplus_{\langle{\alpha},\xi\rangle\leq mN(|x|+1)}	R_{m,\alpha}.
% 	\end{equation}
Consider functions $H_m: \mathbb R^2\rightarrow \mathbb R$ defined by 
\begin{equation}
	H_m(x,y)=\frac{1}{m^n}\dim_{\mathbb C}\left(\mathcal F^{ym}_{\mathrm{wt}_\eta}R_m/\mathcal F^{xm}_{\mathrm{wt}_\xi}R_m\cap\mathcal F^{ym}_{\mathrm{wt}_\eta}R_m\right).
\end{equation} Then we obtain 
$\mathrm{DH}_m(\xi,\eta)=-\frac{\p^2}{\p x\p y}H_m$ in the sense of distributions.
 Using \eqref{dimension estimate} and some elementary arguments, to prove \eqref{eq--volume is differentiable}, it suffices to show that the right-hand side limit in \eqref{eq--volume is differentiable}
exists. It is clear that the limit depends linearly on $\eta$, so we may assume that there exists an $\alpha_\eta\in \Lambda_0$ such that $\langle\alpha_\eta,\eta\rangle>0$, i.e., there exists $f_{\eta}\in R_{\alpha_{\eta}}$ such that 
	\begin{equation}
		\mathrm{wt}_{\eta}(f_{\eta})>0.
	\end{equation}
As in the proof of Proposition \ref{lem-convergence of measures}, we again appeal to Theorem \ref{okounkov body}. One can extend $\mathrm{wt}_\xi$ to be a rank $n$ valuation $\vv{\mathrm {val}}$ and obtain 
\begin{equation}\label{computation for H-m}
     H_m(x,y)=\# \Gamma_m-\# \Gamma'_m,
 \end{equation}where $\Gamma$ and $\Gamma'$ are two sub-semigroups of 
$\mathbb N^{n+1}$ given by
 \begin{equation}\label{subsemigroups}
     \begin{aligned}
         &\Gamma=\{(d_1, \cdots, d_n, m)\mid (d_1, \cdots, d_n)\in\vv{\mathrm{val}}(\mathcal F^{ym}_{\mathrm{wt}_\eta}R_m\cap \bigoplus_{\langle{\alpha},\xi\rangle\leq Nm(|x|+1)}	R_{m,\alpha})\},\\
         &\Gamma'=\{(d_1, \cdots, d_n, m)\mid (d_1, \cdots, d_n)\in \vv{\mathrm{val}}(\mathcal F^{xm}_{\mathrm{wt}_\xi}R_m\cap\mathcal F^{ym}_{\mathrm{wt}_\eta}R_m\cap \bigoplus_{\langle{\alpha},\xi\rangle\leq Nm(|x|+1)}	R_{m,\alpha})\}.
     \end{aligned}
 \end{equation} 
It is clear that they satisfy Condition (1) and (2) of Theorem \ref{okounkov body}. Using $f_\eta$, one can choose $N\gg1$, and argue as in \eqref{eq--holomorphic with zero valuation}--\eqref{eq--generate as a group} to see that Condition (3) also holds.
	Then as in the proof of Proposition \ref{lem-convergence of measures}, the dominated convergence theorem implies the measure $\mathrm{DH}_m(\xi,\eta)$ converges weakly to a unique Radon measure $\mathrm{DH}(\xi,\eta)$ on $\mathbb R^2$. Moreover, by \eqref{eq--support of the measure} the support of $\mathrm{DH}(\xi,\eta)$ is contained in the region 
	\begin{equation*}
 		\{(x,y)\mid x\geq -A\text{ and }  |y|\leq Ax+B\}.
 	\end{equation*}
Then by \eqref{dimension estimate}, one can easily show that the right-hand side limit in \eqref{eq--volume is differentiable}
exists and equals $$-\int_{\mathbb R^2}ye^{-x}\mathrm{DH}(\xi, \eta).$$
  Lastly, we verify that the second derivative in \eqref{eq-second order derivative} is always positive. Indeed since $\eta$ is nonzero, up to replacing $\eta$ by $-\eta$, we can find $n$ algebraically independent sections $s_i\in R_{m_i,\alpha_i}$ such that 
  \begin{equation}
      \langle\frac{\alpha_i}{m},\eta\rangle>0.
  \end{equation}Using these sections, it is easy to see the right-hand side of \eqref{eq-second order derivative} is positive.

\bibliographystyle{alpha}
\bibliography{ref.bib}

\end{document}